\documentclass[9pt,journal]{ieeetran}

\usepackage[T1]{fontenc}
\usepackage{mathptmx}
\usepackage{amsmath}
\usepackage{amssymb}
\usepackage{amsfonts}
\usepackage{amsthm}
\usepackage{mathdots}
\usepackage{oldgerm}

\usepackage{pifont}
\usepackage{comment}
\usepackage{mathdots}
\usepackage{tikz}
\usepackage{graphicx,pst-all,bm}
\usepackage{fixltx2e}
\usepackage{flushend}
\usepackage{subfig}
\usepackage[shortlabels]{enumitem}
\usepackage{caption}
\usepackage{longtable}
\usepackage{verbatim}

\newcommand{\cB}{\mathcal{B}}
\newcommand{\cC}{\mathcal{C}}
\newcommand{\cD}{\mathcal{D}}

\newcommand{\cL}{\mathcal{L}}
\newcommand{\cM}{\mathcal{M}}
\newcommand{\cN}{\mathcal{N}}

\newcommand{\cR}{\mathcal{R}}

\newcommand{\cW}{\mathcal{W}}

\usetikzlibrary{arrows}

\theoremstyle{theorem}
\newtheorem{Prop}{Proposition}[section]
\newtheorem{Lem}[Prop]{Lemma}
\newtheorem{Thm}[Prop]{Theorem}
\newtheorem{Cor}[Prop]{Corollary}

\theoremstyle{definition}
\newtheorem{Def}[Prop]{Definition}
\newtheorem{Rem}[Prop]{Remark}

\newtheorem{Ex}[Prop]{Example}

\newcommand{\N}{{\mathbb{N}}}

\newcommand{\R}{{\mathbb{R}}}
\newcommand{\C}{{\mathbb{C}}}

\newcommand{\Gl}{\mathbf{Gl}}

\DeclareMathOperator{\im}{im}

\DeclareMathOperator{\rk}{rk}
\DeclareMathOperator{\diag}{diag}

\newcommand{\ddt}{\tfrac{\text{\normalfont d}}{\text{\normalfont d}t}}
\newcommand{\ds}[1]{{\rm \, d} #1 \,}
\newcommand{\setdef}[2]{\left\{\ #1\ \left|\ \vphantom{#1} #2\ \right.\right\}}

\DeclareMathOperator{\sgn}{sgn}
\DeclareMathOperator{\loc}{loc}
\DeclareMathOperator{\sat}{sat}
\DeclareMathOperator{\adj}{adj}

\DeclareMathOperator{\erf}{erf}
\DeclareMathOperator{\erfc}{erfc}

\newlength{\innersep}
\newlength{\maxlength}
\newlength{\dummylength}

\newcommand{\JordanBlock}[3]{
\setlength{\arraycolsep}{0pt}
\renewcommand{\arraystretch}{0}
\settowidth{\maxlength}{$#1$}
\settoheight{\dummylength}{$#1$}
\ifdim\dummylength>\maxlength
  \setlength{\maxlength}{\dummylength}
\fi
\settowidth{\dummylength}{$#2$}
\ifdim\dummylength>\maxlength
  \setlength{\maxlength}{\dummylength}
\fi
\settoheight{\dummylength}{$#2$}
\ifdim\dummylength>\maxlength
  \setlength{\maxlength}{\dummylength}
\fi
\setlength{\innersep}{0.1\maxlength}
\addtolength{\maxlength}{\innersep}
\addtolength{\maxlength}{\innersep}
\newcommand{\invisiblebox}{\phantom{\rule{\maxlength}{\maxlength}}}
\begin{array}{ccc}
  {\tikz[remember picture] \node[outer sep=0,inner sep=\innersep] (a11) {$#1$};} & \invisiblebox & {\tikz[remember picture] \node[outer sep=0,inner sep=\innersep] (a13) {$#1$};}\\
   \invisiblebox &\phantom{\rule{#3}{#3}} & \invisiblebox \\
   {\tikz[remember picture] \node[outer sep=0,inner sep=\innersep] (a31) {$#2$};} & \invisiblebox & {\tikz[remember picture] \node[outer sep=0,inner sep=\innersep] (a33) {$#1$};}
\end{array}
\tikz[remember picture, overlay] \draw (a11) edge[very thick] (a33);
}

\newcommand{\UnityBlock}[3]{
\setlength{\arraycolsep}{0pt}
\renewcommand{\arraystretch}{0}
\settowidth{\maxlength}{$#1$}
\settoheight{\dummylength}{$#1$}
\ifdim\dummylength>\maxlength
  \setlength{\maxlength}{\dummylength}
\fi
\settowidth{\dummylength}{$#2$}
\ifdim\dummylength>\maxlength
  \setlength{\maxlength}{\dummylength}
\fi
\settoheight{\dummylength}{$#2$}
\ifdim\dummylength>\maxlength
  \setlength{\maxlength}{\dummylength}
\fi
\setlength{\innersep}{0.1\maxlength}
\addtolength{\maxlength}{\innersep}
\addtolength{\maxlength}{\innersep}
\newcommand{\invisiblebox}{\phantom{\rule{\maxlength}{\maxlength}}}
\begin{array}{ccc}
  {\tikz[remember picture] \node[outer sep=0,inner sep=\innersep] (a11) {$#1$};} & \invisiblebox & \invisiblebox\\
   \invisiblebox &\phantom{\rule{#3}{#3}} & \invisiblebox \\
   \invisiblebox & \invisiblebox & {\tikz[remember picture] \node[outer sep=0,inner sep=\innersep] (a33) {$#1$};}
\end{array}
\tikz[remember picture, overlay] \draw (a11) edge[very thick] (a33);
}

\newcommand{\ReverseUnityBlock}[3]{
\setlength{\arraycolsep}{0pt}
\renewcommand{\arraystretch}{0}
\settowidth{\maxlength}{$#1$}
\settoheight{\dummylength}{$#1$}
\ifdim\dummylength>\maxlength
  \setlength{\maxlength}{\dummylength}
\fi
\settowidth{\dummylength}{$#2$}
\ifdim\dummylength>\maxlength
  \setlength{\maxlength}{\dummylength}
\fi
\settoheight{\dummylength}{$#2$}
\ifdim\dummylength>\maxlength
  \setlength{\maxlength}{\dummylength}
\fi
\setlength{\innersep}{0.1\maxlength}
\addtolength{\maxlength}{\innersep}
\addtolength{\maxlength}{\innersep}
\newcommand{\invisiblebox}{\phantom{\rule{\maxlength}{\maxlength}}}
\begin{array}{ccc}
   \invisiblebox & \invisiblebox & {\tikz[remember picture] \node[outer sep=0,inner sep=\innersep] (a13) {$#1$};}\\
   \invisiblebox &\phantom{\rule{#3}{#3}} & \invisiblebox \\
   {\tikz[remember picture] \node[outer sep=0,inner sep=\innersep] (a31) {$#1$};} &\invisiblebox & \invisiblebox
\end{array}
\tikz[remember picture, overlay] \draw (a13) edge[very thick] (a31);
}

\newcommand{\LowerNilBlock}[3]{
\setlength{\arraycolsep}{0pt}
\renewcommand{\arraystretch}{0}
\settowidth{\maxlength}{$#1$}
\settoheight{\dummylength}{$#1$}
\ifdim\dummylength>\maxlength
  \setlength{\maxlength}{\dummylength}
\fi
\settowidth{\dummylength}{$#2$}
\ifdim\dummylength>\maxlength
  \setlength{\maxlength}{\dummylength}
\fi
\settoheight{\dummylength}{$#2$}
\ifdim\dummylength>\maxlength
  \setlength{\maxlength}{\dummylength}
\fi
\setlength{\innersep}{0.1\maxlength}
\addtolength{\maxlength}{\innersep}
\addtolength{\maxlength}{\innersep}
\newcommand{\invisiblebox}{\phantom{\rule{\maxlength}{\maxlength}}}
\begin{array}{cccc}
  \tikz[remember picture] \node[outer sep=0,inner sep=\innersep] (a11) {$#1$}; &  & \invisiblebox & \invisiblebox\\
  {\tikz[remember picture] \node[outer sep=0,inner sep=\innersep] (a21) {$#2$};}&  & \invisiblebox & \invisiblebox\\
   & \phantom{\rule{#3}{#3}} &  & \\
   \invisiblebox & &  {\tikz[remember picture] \node[outer sep=0,inner sep=\innersep] (a43) {$#2$};} &
   {\tikz[remember picture] \node[outer sep=0,inner sep=\innersep] (a44) {$#1$};}
\end{array}
\tikz[remember picture, overlay] \draw (a11) edge[very thick] (a44);
\tikz[remember picture, overlay] \draw (a21) edge[very thick] (a43);
}

\newcommand{\RectBlock}[3]{
\setlength{\arraycolsep}{0pt}
\renewcommand{\arraystretch}{0}
\settowidth{\maxlength}{$#1$}
\settoheight{\dummylength}{$#1$}
\ifdim\dummylength>\maxlength
  \setlength{\maxlength}{\dummylength}
\fi
\settowidth{\dummylength}{$#2$}
\ifdim\dummylength>\maxlength
  \setlength{\maxlength}{\dummylength}
\fi
\settoheight{\dummylength}{$#2$}
\ifdim\dummylength>\maxlength
  \setlength{\maxlength}{\dummylength}
\fi
\setlength{\innersep}{0.1\maxlength}
\addtolength{\maxlength}{\innersep}
\addtolength{\maxlength}{\innersep}
\newcommand{\invisiblebox}{\phantom{\rule{\maxlength}{\maxlength}}}
\begin{array}{ccccc}
  \tikz[remember picture] \node[outer sep=0,inner sep=\innersep] (a11) {$#1$}; &{\tikz[remember picture] \node[outer sep=0,inner sep=\innersep] (a21) {$#2$};}& & \invisiblebox & \invisiblebox\\
  &&\phantom{\rule{#3}{#3}} &&\\
   \invisiblebox &\invisiblebox& &
   {\tikz[remember picture] \node[outer sep=0,inner sep=\innersep] (a44) {$#1$};} & {\tikz[remember picture] \node[outer sep=0,inner sep=\innersep] (a43) {$#2$};}
\end{array}
\tikz[remember picture, overlay] \draw (a11) edge[very thick] (a44);
\tikz[remember picture, overlay] \draw (a21) edge[very thick] (a43);
}

\renewcommand{\theenumi}{{\rm(\roman{enumi})}} 

\newenvironment{smallpmatrix}
{\left(\begin{smallmatrix}}
{\end{smallmatrix}\right)}

\newenvironment{smallbmatrix}
{\left[\begin{smallmatrix}}
{\end{smallmatrix}\right]}

\makeatletter
\renewcommand*\env@matrix[1][*\c@MaxMatrixCols c]{%
  \hskip -\arraycolsep
  \let\@ifnextchar\new@ifnextchar
  \array{#1}}
\makeatother

\begin{document}

\title{Fault tolerant funnel control for uncertain linear systems\thanks{This work was supported by the German Research Foundation (Deutsche Forschungsgemeinschaft) via the grant BE 6263/1-1. Furthermore, the author thanks Achim Ilchmann (Technische Universtit\"at Ilmenau) for several constructive discussions.}}

\author{Thomas Berger
\thanks{Thomas Berger is with the Institut f\"ur Mathematik, Universit\"at Paderborn, Warburger Str.~100, 33098~Paderborn, Germany, {\tt\small thomas.berger@math.upb.de}}}

\date{\today}
\maketitle

%
%
\begin{abstract}
We study adaptive fault tolerant tracking control for uncertain linear systems. Based on recent results in funnel control and the time-varying Byrnes-Isidori form, we develop a low-complexity model-free controller which achieves prescribed performance of the tracking error for any given sufficiently smooth reference signal. Within the considered system class, we allow for more inputs than outputs as long as a certain redundancy of the actuators is satisfied. An important role in the controller design is played by the controller weight matrix. This is a rectangular input transformation chosen such that in the resulting system the zero dynamics, which are assumed to be uniformly exponentially stable, are independent of the new input. We illustrate the fault tolerant funnel controller by an example of a linearized model for the lateral motion of a Boeing~737 aircraft.
\end{abstract}

%
\section{Introduction}\label{Sec:Intr}
%

Being able to handle system uncertainties and, at the same time, failures or degrading efficiency of actuators is an important task in the design of control techniques. There are basically four different research directions in fault tolerant control, see the nice literature survey in~\cite{TaoChen04}. These are (a)~multiple-model, switching, and tuning, (b)~direct and indirect adaptive designs, (c)~fault detection and diagnosis, and (d)~robust control design. We also refer to the exhaustive review paper~\cite{ZhanJian08} and the recent surveys~\cite{GaoCeca15a, GaoCeca15b} for more references.

The uncertainties and actuator faults appearing in the system are usually unknown both in their nature and extent. In this framework, an adaptive control approach seems a suitable choice. The fault tolerant funnel controller that we introduce in the present paper is such a direct adaptive design. The area of adaptive design methods is quite active, see the recent articles~\cite{DengYang16,XieYang16a,ZhanYang17b}. Different approaches have been pursued, such as filter design and backstepping~\cite{DengYang16}, strategies based on solving optimal control problems~\cite{LiYang12,XieYang16a,YangYang00} and (model-free) adaptive control techniques~\cite{TaoChen04,TaoJosh01,ZhanYang17b}.

In the present paper we consider adaptive fault tolerant tracking control for uncertain linear systems with prescribed performance of the tracking error. The uncertainties incorporate modelling errors and process faults as well as bounded noises and disturbances. The actuator faults encompass possible failures and degrading efficiency of the actuators as well as actuator stuck, locked actuator faults, actuator bias and actuator saturation. In the literature, some types of faults are often excluded; in~\cite{ZhanYang17b} no total faults are allowed, in~\cite{TaoChen04, TaoJosh01} only actuator stuck is considered, and actuator saturation is considered in none of the aforementioned works.

Most results in fault tolerant control are model-based, cf.~\cite{GaoCeca15a, XieYang16a}. The approach presented in~\cite{ZhanYang17b} is completely model-free, however only single-input, single-output systems with trivial internal dynamics are considered and total faults are excluded. The approaches in~\cite{TaoChen04, TaoJosh01} require only little knowledge about the system parameters.

As the first result in fault tolerant tracking control that the author is aware of, the design in~\cite{ZhanYang17b} is able to achieve prescribed performance of the tracking error and it is based on the approach of \emph{Prescribed Performance Control} developed in~\cite{BechRovi08}, see also~\cite{BechRovi14}. However, in the present paper we follow the complementary approach of \emph{Funnel Control} which was developed in~\cite{IlchRyan02b}, see also the survey~\cite{IlchRyan08} and the references therein. The funnel controller is an adaptive controller of high-gain type and thus inherently robust, which makes it a suitable choice for fault tolerant control tasks. The funnel controller has been successfully applied e.g.\ in temperature control of chemical reactor models~\cite{IlchTren04}, control of industrial servo-systems~\cite{Hack17} and underactuated multibody systems~\cite{BergOtto19}, DC-link power flow control~\cite{SenfPaug14}, voltage and current control of electrical circuits~\cite{BergReis14a}, control of peak inspiratory pressure~\cite{PompWeye15} and adaptive cruise control~\cite{BergRaue18}.

Since it is usually not possible to foresee which actuator may fail during the operation of a system, a certain redundancy of the actuators is required, so that the remaining actuators are able to compensate for the (total) fault of others. Therefore, a larger number of actuators than sensors is required, which leads to systems with more inputs than outputs and thus additionally complicates the control task. For instance, funnel control has only been investigated for systems with the same number of inputs and outputs, see e.g.~\cite{IlchRyan08,IlchRyan02b}. The funnel control design that we introduce in the present paper extends the recently developed funnel controller for systems with arbitrary relative degree~\cite{BergLe18a}.

We provide extensions of the above mentioned results for uncertain linear systems in the following regard:
\begin{itemize}
  \item the allowed uncertainties and actuator faults encompass essentially all relevant cases,
  \item the control design is model-free,
  \item more inputs than outputs are allowed, which in particular extends available results in funnel control,
  \item the relative degree of the system may be arbitrary, but known, and the zero dynamics may be nontrivial,
  \item prescribed performance of the tracking error is achieved,
  \item the controller is simple in its design and of low complexity.
\end{itemize}


Throughout this article, we use the following notation: We write $\R_{\ge 0}=[0,\infty)$ and $\C_-$, $(\C_+)$ denotes the set of complex numbers with negative (positive) real part. $ {\Gl}_n(\R)$ denotes the group of invertible matrices in $\R^{n\times n}$, $\sigma(A)$ the spectrum of $A\in\R^{n\times n}$, and $M^\dagger$ and $\rk M$ the Moore-Penrose pseudoinverse and rank of $M\in\R^{n\times m}$, resp. By $\mathcal{L}^\infty(I\!\to\!\R^n)$ we denote the set of essentially bounded functions $f:I\!\to\!\R^n$ with norm $\|f\|_\infty = {\rm ess\ sup}_{t\in I} \|f(t)\|$. The set $\mathcal{L}_{\loc}^\infty(I\!\to\!\R^n)$ contains all locally essentially bounded functions and $\mathcal{W}^{k,\infty}(I\!\to\!\R^n)$ is the set of $k$-times weakly differentiable functions $f:I\!\to\!\R^n$ such that $f,\ldots, f^{(k)}\in \mathcal{L}^\infty(I\!\to\!\R^n)$. By $\mathcal{C}^k(I\!\to\!\R^n)$ we denote the set of $k$-times continuously differentiable functions, where $k\in\N_0\cup\{\infty\}$, and we use $\mathcal{C}(I\!\to\!\R^n) = \mathcal{C}^0(I\!\to\!\R^n)$. Finally, $\left.f\right|_{J}$ denotes the restriction of the function $f:I\!\to\!\R^n$ to $J\subseteq I$.

\vspace*{-3mm}

\subsection{System class}\label{Ssec:SysClass}

In the present paper we consider linear systems with time-varying and nonlinear uncertainties and possible actuator faults of the form
\begin{equation}\label{eq:ABC}
\begin{aligned}
  \dot x(t) &= Ax(t) + BL(t) u(t) + f\big(t,x(t),u(t)\big),\quad x(0)=x^0\\
  y(t) &= Cx(t)
\end{aligned}
\end{equation}
where $x^0\in\R^n$, $A\in\R^{n\times n}$, $B\in\R^{n\times m}$, $C\in\R^{p\times n}$ with $m\geq p$, $f\in \mathcal{C}(\R\times \R^n\times\R^m\rightarrow \R^n)$ is bounded and the following properties are satisfied:
\begin{itemize}
\item[(P1)] $L\in\cC^\infty(\R\to\R^{m\times m})$ such that $L, \dot L,\ldots, L^{(n)}$ are bounded and there exists $q\in\N$ such that $\rk BL(t) = q \ge p$ for all $t\in\R$;
\item[(P2)] the system has (strict) relative degree $r\in\N$, i.e.,
    \begin{itemize}[$\bullet$,leftmargin=0.5cm]
    \item $CA^k B L(\cdot) = 0$ and $CA^kf(\cdot)= 0$ for all $k=0,\ldots,r-2$ and
    \item the ``high-frequency gain matrix'' $\Gamma:=CA^{r-1}B\in\R^{p\times m}$ and~$L$ satisfy $\rk \Gamma L(t)= p$ for all $t\in\R$.
    \end{itemize}
\end{itemize}
The functions $u:\R_{\ge 0}\rightarrow\R^m$ and $y:\R_{\ge 0}\rightarrow\R^p$ are called {\it input} and {\it output} of the system~\eqref{eq:ABC}, resp. Some comments on the system class~\eqref{eq:ABC} are warranted.
\begin{enumerate}
\item The control objective is fault tolerant control (see Subsection~\ref{Ssec:ContrObj}), hence a certain redundancy of the actuators is necessary in~\eqref{eq:ABC}, i.e.,~$m$ is usually much larger than~$p$.  The (unknown) matrix-valued function~$L$ from~(P1) describes the \emph{reliability} of the actuators. Typically we have $L(t) = \diag( l_1(t),\ldots,l_m(t))$ with $l_i\in\cC^\infty(\R\to [0,1])$ monotonically non-increasing and $l_i(0)=1$; in this way, possible failures and degrading efficiency of the actuators may be described, cf.\ also~\cite{LiYang12,XieYang16a, YangYang00}. In our framework, we allow for a general smooth matrix-valued function such that $\rk BL(t) = q$ for all $t\in\R$; one may think of~$q$ groups of actuators, where actuators in the same group perform the same control task, and it is assumed that in each group at least one actuator remains (partially) functional. If $q=p$, then this situation is close to the concept of \emph{uniform actuator redundancy}, see~\cite{ZhaoJian98}. Clearly, $\rk BL(t) \geq p$ is necessary for the application of adaptive control techniques.
\item The (unknown) nonlinearity~$f$ describes possible modelling errors or process faults, uncertainties, bounded noises and disturbances, and types of actuator failures not covered by the matrix function~$L$, see e.g.~\cite{Gert88}. The latter means for instance locked actuator faults, actuator bias or actuator saturation, i.e., $f(u_i) = \sat(u_i) + b_i$ with $i\in\{1,\ldots,m\}$, $b_i\in\R$ and $\sat(u_i) = \sgn(u_i)\, \hat u_i$ for $|u_i|\ge \hat u_i$ and $\sat(u_i) = u_i$ for $|u_i| < \hat u_i$, cf.~\cite{DeLuMatt03}.
\item The conditions in~(P2) are slightly stronger than the assumption of a \emph{strict and uniform relative degree $r\in\N$} as introduced for time-varying nonlinear systems with $m=p$ in~\cite[Def.~2.2]{IlchMuel07}. We use the stronger concept, and call it (strict) relative degree again, since, in view of~(iii), we do not want to impose any differentiability assumptions on the nonlinearity~$f$ which are required in~\cite{IlchMuel07}. For the linear part of~\eqref{eq:ABC}, i.e., $f(\cdot)=0$ and $L(\cdot)=I_m$, the notion of strict relative degree as in (P2) is justified (note that $m>p$ is possible) since by~\cite[Def.~B.1]{Berg16b} the transfer function $G(s) = C(sI-A)^{-1} B$ has vector relative degree $(r,\ldots,r)$, cf.\ also~\cite{Isid95, Muel09a}.
\item We assume that the system parameters $A$, $B$, $C$, $L(\cdot)$, $f(\cdot)$, $x^0$ are unknown; in particular the state space dimension~$n$ does not need to be known. We only require knowledge of the relative degree~$r\in\N$. Furthermore, we will derive a class of rectangular input transformations of the form $u(t) = K(t) v(t)$, where $K\in\cC^\infty(\R\to\R^{m\times p})$, such that in the resulting system the zero dynamics are independent of the new input~$v$. As a structural assumption, we will require that the zero dynamics of the time-varying linear system $(A,BL(\cdot)K(\cdot),C)$ are uniformly exponentially stable for one (and hence any)~$K$ in this class; it is hence independent of the choice of~$K$. Some additional knowledge of system parameters, such as the high-frequency gain matrix $\Gamma = C A^{r-1} B$ from~(P2), may be helpful for the construction of~$K$, while it is not required; see Subsection~\ref{Ssec:ConGain}.
\end{enumerate}

We stress that even in the case $L(\cdot) = I_m$ the results of the present paper are new when $m>p$, since $m=p$ is usually assumed in funnel control.

\vspace*{-3mm}

\subsection{Control objective}\label{Ssec:ContrObj}

The objective is fault tolerant tracking of a reference trajectory $y_{\rm ref}\in\mathcal{W}^{r,\infty}(\R_{\ge 0}\to\R^p)$ with prescribed performance, i.e., we seek an output error derivative feedback such that in the closed-loop system the tracking error $e(t) = y(t)-y_{\rm ref}(t)$ evolves within a prescribed performance funnel
\begin{equation}
\mathcal{F}_{\varphi} := \setdef{(t,e)\in\R_{\ge 0} \times\R^p}{\varphi(t) \|e\| < 1},\label{eq:perf_funnel}
\end{equation}
which is determined by a function~$\varphi$ belonging to
\begin{equation*}
\Phi_r \!:=\!
\left\{
\varphi\in  \cC^r(\R_{\ge 0}\to\R)
\left|\!\!\!
\begin{array}{l}
\text{ $\varphi, \dot \varphi,\ldots,\varphi^{(r)}$ are bounded,}\\
\text{ $\varphi (\tau)>0$ for all $\tau>0$,}\\
 \text{ and }  \liminf_{\tau\rightarrow \infty} \varphi(\tau) > 0
\end{array}
\right.\!\!\!
\right\}.
\label{eq:Phir}\end{equation*}
Furthermore, the state~$x$ and the input~$u$ in~\eqref{eq:ABC} should remain bounded.

The funnel boundary is given by the reciprocal of~$\varphi$, see Fig.~\ref{Fig:funnel}. If $\varphi(0)=0$, then no restriction is put on the initial error $e(0)$. Furthermore, each performance funnel $\mathcal{F}_{\varphi}$ with $\varphi\in\Phi_r$ is bounded

\vspace{1mm}
\begin{minipage}{.25\textwidth}
\hspace*{-8mm}\includegraphics[width=6.4cm, trim=170 550 230 120, clip]{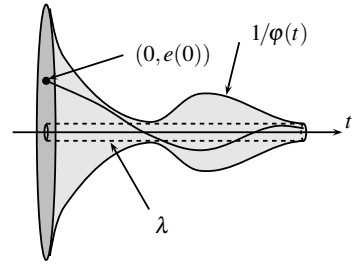}
\captionof{figure}{Error evolution in a funnel $\mathcal F_{\varphi}$  with boundary $\varphi(\cdot)^{-1}$.}
\label{Fig:funnel}
\end{minipage}
\hspace{1mm}
\begin{minipage}{.2\textwidth}
 away from zero, since boundedness of $\varphi$ gives $1/\varphi(t)\geq\lambda$ for all $t > 0$ and some $\lambda>0$. While it is often convenient to choose a monotonically decreasing funnel boundary, it might be advantageous to widen the funnel over some later time interval, for instance in the presence of periodic disturbances or strongly varying reference signals.
\end{minipage}



\subsection{Organization of the present paper}\label{Ssec:Contrib}

The paper is structured as follows. In Section~\ref{Sec:NF} we derive a normal form for system~\eqref{eq:ABC} which extends the Byrnes-Isidori form for time-varying linear systems from~\cite{IlchMuel07}. We derive a class of rectangular input transformations such that in the resulting system the zero dynamics are independent of the new input and uniformly exponentially stable. The rectangular input transformation is exploited as controller weight matrix in the design of a fault tolerant funnel controller in Section~\ref{Sec:FTC} and possible choices are discussed. The performance of the proposed funnel controller is illustrated by means of a linearized model for the lateral motion of a Boeing~737 aircraft in Section~\ref{Sec:Sim}. 

\vspace*{-3mm}

%
\section{A time-varying normal form}\label{Sec:NF}
%


We derive a normal form for systems~\eqref{eq:ABC} which is an extension of the Byrnes-Isidori form for time-varying linear systems from~\cite{IlchMuel07}. In this paper, with ``normal form'' we do not mean a ``canonical form'' which would be a unique representative of its equivalence class with respect to a certain set of transformations (or the mapping to this representative, resp.), but rather a weaker notion. We will see that the freedom left within the non-zero entries of the derived decomposition of~\eqref{eq:ABC} can be specified and is not significant which justifies to call it ``normal form''.

We introduce the following matrix-valued functions:
\begin{align*}
    \cB(t) &:= \left[BL(t),\left(\ddt\!-\!A\right)\!\big(BL(t)\big), \ldots, \left(\ddt\!-\!A\right)^{r-1}\!\!\big(BL(t)\big)\right]\!\in\!\R^{n\times rm},\\
    \cC &:= \left[ C^\top, (CA)^\top, \ldots, \big(CA^{r-1}\big)^\top \right]^\top \in\R^{rp \times n},\quad t\in\R.
\end{align*}
Let $\rho:= \rk \cC$, choose $V\in\R^{n\times (n-\rho)}$ such that $\im V = \ker \cC$ and define
\[
    \cN(t) := V^\dagger \left[ I_n - \cB(t) \big(\cC \cB(t)\big)^\dagger \cC\right]\in\R^{(n-\rho)\times n},\quad t\in\R,
\]
as well as
\begin{equation}\label{eq:U(t)}
    U(t) := \begin{bmatrix} \cC \\ \cN(t)\end{bmatrix} \in\R^{(n-\rho+pr)\times n},\quad t\in\R.
\end{equation}

\begin{Lem}\label{Lem:invU}
Consider a system~\eqref{eq:ABC} with (P1) and (P2). Then we have for all $t\in\R$ that $\rho = \rk \cC = \rk \cC \cB(t) = pr$ and
\begin{equation}\label{eq:CB}
    \cC\cB(t) = \begin{smallbmatrix} 0 & & (-1)^{r-1} \Gamma L(t)\\  & \reflectbox{$\ddots$} & \\ \Gamma L(t) & & \ast \end{smallbmatrix}.
\end{equation}
Furthermore, $U(\cdot)$ as in~\eqref{eq:U(t)} is invertible with
\begin{equation}\label{eq:U(t)inv}
    U(t)^{-1} = \left[ \cB(t) \big(\cC \cB(t)\big)^\dagger, V\right],\quad t\in\R.
\end{equation}
\end{Lem}
\begin{proof} Similar to the proof of~\cite[Prop.~3.1]{IlchMuel07} and using~(P2) it is straightforward to show that~\eqref{eq:CB} holds. Hence $\rk \cC \cB(t) = pr$ for all $t\in\R$ since $\rk \Gamma L(t) = p$. Therefore, $\cC \cB(t)  \big(\cC \cB(t)\big)^\dagger = I_{pr}$, which implies
$\begin{smallbmatrix} \cC \\ \cN(t)\end{smallbmatrix} \left[ \cB(t) \big(\cC \cB(t)\big)^\dagger, V\right] = \begin{smallbmatrix} I_{pr} & 0 \\ 0 & I_{n-\rho}\end{smallbmatrix}$, thus $n \ge \rk U(t) = \rk \begin{smallbmatrix} \cC \\ \cN(t)\end{smallbmatrix} = n-\rho+pr$ for all $t\in\R$. As a consequence, $pr \le \rho = \rk \cC \le pr$ and this shows $\rho = \rk \cC = pr$ and~\eqref{eq:U(t)inv}.
\end{proof}

$U$ as in~\eqref{eq:U(t)} will serve as a time-varying state space transformation in the following. Therefore, it will be important that~$U$ is a Lyapunov transformation.

\begin{Def}
  We call $M\in\cC^1(\R\to\Gl_n(\R))$ a \emph{Lyapunov transformation}, if $M$, $M^{-1}$ and $\dot M$ are bounded.
\end{Def}

By~\eqref{eq:U(t)} and~\eqref{eq:U(t)inv} it is straightforward to see that~$U$ is a Lyapunov transformation if, and only if, $\cB \big(\cC\cB\big)^\dagger$  and $\ddt \left(\cB \big(\cC\cB\big)^\dagger\right)$ are bounded. A simpler condition, which however is only sufficient, is given in the next result.

\begin{Lem}\label{Lem:boundU}
Consider a system~\eqref{eq:ABC} with (P1) and (P2). Then~$U$ as in~\eqref{eq:U(t)} is a Lyapunov transformation, if
\begin{equation}\label{eq:cond-boundU}
    \exists\,\alpha>0\ \forall\, t\in\R:\ \det \Big(\cC \cB(t) \big(\cC \cB(t)\big)^\top\Big) \geq \alpha.
\end{equation}
The converse implication is false in general.
\end{Lem}
\begin{proof} First we show that a Lyapunov transformation~$U$ does not necessarily satisfy~\eqref{eq:cond-boundU}. To this end, consider $A=0$, $B=C=1$, $f=0$ and $L(t) = \tfrac{1}{t^2+1}$. Then~(P1) and~(P2) are satisfied with $r=1$ and $U=1$ is a Lyapunov transformation, but~\eqref{eq:cond-boundU} is not satisfied.

It remains to show that~\eqref{eq:cond-boundU} implies that~$U$ is a Lyapunov transformation. This can be inferred using that $\cB$ and $\dot \cB$ are bounded by~(P1) and that for any pointwise invertible $M\in\cC^\infty(\R\to\R^{k\times k})$ we have $M(t)^{-1} = \frac{\adj M(t)}{\det M(t)}$,
\[
    \ddt M(t)^{-1} = \frac{\big(\ddt \adj M(t)\big) \det M(t) - \adj M(t) \big(\ddt \det M(t)\big)}{\left(\det M(t)\right)^2}. \qedhere
\]
\end{proof}

Note that, in view of Lemma~\ref{Lem:invU}, $\det \Big(\cC \cB(\cdot) \big(\cC \cB(\cdot)\big)^\top\Big)$ is a positive and smooth function and condition~\eqref{eq:cond-boundU} only requires that, roughly speaking, it does not decay to zero for $t\to\infty$ (or $t\to -\infty$, if this is of relevance).

A crucial tool for the proof of the main result of this section is the following relation.

\begin{Lem}\label{Lem:Bdag=Wdag}
Consider a system~\eqref{eq:ABC} with (P1) and (P2) such that $q=p$ in~(P1). Then
\begin{equation}\label{eq:B=BCBdagCB}
\begin{aligned}
    \forall\, t\in\R:\quad 
    \Big( I_n - \cB(t) \big(\cC \cB(t)\big)^\dagger \cC\Big) \cB(t) = 0
\end{aligned}
\end{equation}
and, as a consequence, $\cN(t) \cB(t) = 0$ for all $t\in\R$.
\end{Lem}
\begin{proof} First observe that by property~(P1) and~\cite[Thm.~3.9]{KunkMehr06} there exists $W\in\cC^\infty(\R\to\R^{n\times q})$ with $\rk W(t)=q$ for all $t\in\R$ and pointwise orthogonal $[R_1,R_2]\in\cC^\infty(\R\to\R^{m\times (q + (m-q))})$ such that
\begin{equation}\label{eq:kerBL}
 \forall\, t\in\R:\   BL(t) [R_1(t), R_2(t)] = [W(t), 0].
\end{equation}
Define, for all $t\in\R$,
\begin{align*}
    \cW(t) & := \left[W(t),\left(\ddt - A\right) W(t), \ldots, \left(\ddt - A\right)^{r-1} W(t)\right],\\
    \cR(t) &:= \left( \cR_{i,j} \right)_{i,j=1,\ldots,r},\quad  \cR_{i,j}(t) := \left\{\begin{array}{rl} \tbinom{j-1}{i-1} R_1^{(j-i)}(t)^\top, &\quad j\geq i,\\ 0, &\quad j<i.\end{array}\right.
\end{align*}
\emph{Step 1:} Fix $t\in\R$. We show that $\cW(t)\cR(t) = \cB(t)$ by proving that
\[
    \cW_j(t) \! := \! \cW(t)\cR(t) \begin{smallbmatrix} 0_{(j-1)m\times m}\\ I_m\\ 0_{(r-j)m\times m}\end{smallbmatrix} \! = \! \cB(t) \begin{smallbmatrix} 0_{(j-1)m\times m}\\ I_m\\ 0_{(r-j)m\times m}\end{smallbmatrix} \! =: \! \cB_j(t)
\]
for all $j=1,\ldots,r$. For any $j\in\{1,\ldots,r\}$ we have
\[
    \cW_j(t) = \sum\nolimits_{i=1}^j \big[ (\ddt-A)^{i-1} W(t)\big] \tbinom{j-1}{i-1} R_1^{(j-i)}(t)^\top
\]
and invoking formula~\cite[(3.2)]{IlchMuel07} with $C=I$ and $i=0$ it follows that
\[
    \cW_j(t) = \sum\nolimits_{i=0}^{j-1} \sum\nolimits_{k=0}^i (-1)^{k} \tbinom{i}{k} \tbinom{j-1}{i} A^k W^{(i-k)}(t) R_1^{(j-i-1)}(t)^\top.
\]
Changing the summation over the ``triangle'' in the double sum we obtain
\begin{align*}
   & \cW_j(t) = \sum\nolimits_{k=0}^{j-1} \sum\nolimits_{i=k}^{j-1} (-1)^{k} \tbinom{i}{k} \tbinom{j-1}{i} A^k W^{(i-k)}(t) R_1^{(j-i-1)}(t)^\top\\
    &= \sum\nolimits_{k=0}^{j-1} \sum\nolimits_{l=0}^{j-k-1} (-1)^{k} \tbinom{j-1}{l+k} \tbinom{l+k}{k} A^k W^{(l)}(t) R_1^{(j-l-k-1)}(t)^\top\\
    &= \sum\nolimits_{k=0}^{j-1} (-1)^{k} \tbinom{j-1}{k} \sum\nolimits_{l=0}^{j-k-1} \tbinom{j-k-1}{l} A^k W^{(l)}(t) R_1^{(j-l-k-1)}(t)^\top\\
    &= \sum\nolimits_{k=0}^{j-1} (-1)^{k} \tbinom{j-1}{k} \big(\ddt\big)^{j-k-1} \left[ A^k W(t) R_1(t)^\top\right],
\end{align*}
where we used $\tbinom{j-1}{l+k} \tbinom{l+k}{k} = \tbinom{j-1}{k} \tbinom{j-k-1}{l}$. Again using formula~\cite[(3.2)]{IlchMuel07} we finally obtain
\begin{align*}
    \cW_j(t) &= (\ddt-A)^{j-1} \big[ W(t) R_1(t)^\top\big] = (\ddt-A)^{j-1} \big[ B L(t)\big] = \cB_j(t).
\end{align*}

\emph{Step 2}: We show that $\rk \cB(t) = pr$ for all $t\in\R$. It follows from~(P2) and~\eqref{eq:kerBL} that $CA^{k} W(t) = 0$ for $k=0,\ldots,r-2$ and $\rk CA^{r-1} W(t) = p$ for all $t\in\R$, by which, using $q=p$, $CA^{r-1} W(t) \in\Gl_{p}(\R)$. Therefore,~\cite[Cor.~3.3]{IlchMuel07} implies that $\rk \cW(t) = pr$ for all $t\in\R$. Furthermore, it is clear that $\rk \cR(t) = pr$ for all $t\in\R$ and hence we may infer from Sylvester's rank inequality that $pr  = \rk \cW(t) + \rk \cR(t) - pr \le \rk \cW(t) \cR(t) \le \min\{\rk \cW(t), \rk \cR(t)\} = pr$. This shows $\rk \cB(t) = \rk \cW(t)\cR(t) = pr$ for all $t\in\R$.

\emph{Step 3}: We show the assertion of the lemma. Since~$\cB$ has constant rank we may again apply~\cite[Thm.~3.9]{KunkMehr06} to find $Y\in\cC^\infty(\R\to\R^{n\times pr})$ with $\rk Y(t)=pr$ and pointwise orthogonal $[V_1,V_2]\in\cC^\infty(\R\to\R^{rm\times (pr + (m-p)r)})$ such that $\cB(t) [V_1(t), V_2(t)] = [Y(t), 0]$ for all $t\in\R$. Then using that $\cB(t) = Y(t) V_1(t)^\top$ and that $[V_1, V_2]$ is pointwise orthogonal, by which $V_1(t)^\top V_1(t) = I_{pr}$ for all $t\in\R$, we obtain from a straightforward computation that $\cB(t)  \big(\cC \cB(t)\big)^\dagger = Y(t) \big(\cC Y(t)\big)^\dagger$. Clearly, $\rk \cC Y(t) = \rk \cC\cB(t) [V_1(t), V_2(t)] = \rk \cC\cB(t) = pr$, thus $\cC Y(t) \in{\Gl}_{pr}(\R)$ for all $t\in\R$. Therefore, it finally follows that
\begin{align*}
  \cB(t)  \big(\cC \cB(t)\big)^\dagger \cC \cB(t) &= Y(t) \big(\cC Y(t)\big)^{-1} \cC Y(t) V_1(t)^\top\\
   &= Y(t) V_1(t)^\top = \cB(t).\qedhere
\end{align*}
\end{proof}

We are now in the position to state the main result on the time-varying normal form.

\begin{Thm}\label{Thm:BIF}
Consider a system~\eqref{eq:ABC} with (P1) and (P2) such that~$U$ as in~\eqref{eq:U(t)} is a Lyapunov transformation. Then
\begin{equation}\label{eq:ABC-trafo}
    (\hat A, \hat B, \hat C) := \big( (UA+\dot U)U^{-1}, UBL, CU^{-1}\big)
\end{equation}
and $\hat f(t,z,u) := U(t) f\big(t,U(t)^{-1}z,u\big)$, $(t,z,u)\in\R^{1+n+m}$, satisfy
\begin{equation}\label{eq:BIF}
\begin{aligned}
    \hat A(t) &= \begin{smallbmatrix} 0 & I_p & 0 & \cdots & 0 & 0\\ 0 & 0& I_p & & & 0\\ \vdots & \vdots & \ddots & \ddots & & \vdots\\ 0 & 0 & \cdots & 0 & I_p & 0\\ R_1(t) & R_2(t) & \cdots & R_{r-1}(t) & R_r(t) & S(t)\\ P_1(t) & P_2(t) & \cdots & P_{r-1}(t) & P_r(t) & Q(t)\end{smallbmatrix},\ \
     \hat B(t) = \begin{smallbmatrix} 0\\ 0 \\ \vdots\\ 0\\ \Gamma L(t)\\ N(t)\end{smallbmatrix},\\
     \hat C &= \begin{bmatrix} I_p & 0 & \cdots  &0\end{bmatrix},\ \ \hat f(t,z,u) = \big( 0, \ldots, 0, f_r(t,z,u), f_\eta(t,z,u)\big)^\top
\end{aligned}
\end{equation}
where $R_i\in\cC^\infty(\R\to\R^{p\times p})$, $P_i, S^\top\in\cC^\infty(\R\to\R^{(n-pr)\times p})$, $Q\in\cC^\infty(\R\to\R^{(n-pr)\times (n-pr)})$, $N\in\cC^\infty(\R\to\R^{(n-pr)\times m})$, $f_r\in\cC(\R\times\R^n\times\R^m\to\R^p)$, $f_\eta\in\cC(\R\times\R^n\times\R^m\to\R^{n-pr})$ are all bounded. Furthermore, the following holds true:
\begin{enumerate}
  \item There exists $K\in\cC^\infty(\R\to\R^{m\times p})$ such that $\Gamma L(t) K(t) = I_p$ and $N(t) K(t) = 0$ for all $t\in\R$ if, and only if,
  \begin{equation}\label{eq:cond-K-Gamma}
    \forall\, t\in\R:\ \im \cB(t) \big(\cC\cB(t)\big)^\dagger \begin{smallbmatrix} 0 \\ I_p\end{smallbmatrix} \subseteq \im BL(t);
  \end{equation}
  in this case we may choose
  \begin{equation}\label{eq:K(t)=}
    K(t) := \big(BL(t)\big)^\dagger \cB(t) \big(\cC\cB(t)\big)^\dagger \begin{smallbmatrix} 0 \\ I_p\end{smallbmatrix}.
  \end{equation}
  \item If $q=p$ for~$q$ in (P1), then $P_2 = P_3 = \ldots = P_r = 0$ and $N=0$.
\end{enumerate}
\end{Thm}
\begin{proof} By the choice of~$U$ as in~\eqref{eq:U(t)} and the representation of its inverse as in Lemma~\ref{Lem:invU} it follows immediately that $\hat A, \hat B, \hat C, \hat f$ have the structure as in the statement of the theorem, cf.\ also~\cite[Thm.~3.5]{IlchMuel07}. Boundedness of all entries follows from the fact that~$U$ is a Lyapunov transformation. It remains to show~(i) and~(ii).

(i): Existence of~$K$ with the mentioned properties is, in view of~\eqref{eq:ABC-trafo}, equivalent to
\begin{equation}
    BL(t) K(t) = U(t)^{-1} \begin{smallbmatrix} 0_{p(r-1)\times p}\\ I_p\\ 0_{(n-pr)\times p}\end{smallbmatrix} \stackrel{\eqref{eq:U(t)inv}}{=} \cB(t) \big(\cC\cB(t)\big)^\dagger \begin{smallbmatrix} 0_{p(r-1)\times p}\\ I_p\end{smallbmatrix} \label{eq:BLK}
\end{equation}
for some~$K\in\cC^\infty(\R\to\R^{m\times p})$. Clearly,~\eqref{eq:BLK} implies~\eqref{eq:cond-K-Gamma}. Conversely, if~\eqref{eq:cond-K-Gamma} holds, then
\begin{equation}\label{eq:BLX}
    \exists\, X:\R\to\R^{m\times pr}:\ \cB(t) \big(\cC\cB(t)\big)^\dagger \begin{smallbmatrix} 0 \\ I_p\end{smallbmatrix}= B L(t) X(t)
\end{equation}
and hence~$K$ as in~\eqref{eq:K(t)=} satisfies
\begin{align*}
     B L(t) K(t) & \stackrel{\eqref{eq:K(t)=}}{=} B L(t) \big(BL(t)\big)^\dagger \cB(t) \big(\cC\cB(t)\big)^\dagger \begin{smallbmatrix} 0 \\ I_p\end{smallbmatrix} \\
     &\stackrel{\eqref{eq:BLX}}{=} B L(t) \big(BL(t)\big)^\dagger B L(t) X(t) \\
     &= BL(t) X(t) \stackrel{\eqref{eq:BLX}}{=} \cB(t) \big(\cC\cB(t)\big)^\dagger \begin{smallbmatrix} 0 \\ I_p\end{smallbmatrix}.
\end{align*}
In view of~\eqref{eq:BLK}, this finishes the proof of~(i).

(ii): If $q=p$, then Lemma~\ref{Lem:Bdag=Wdag} yields that $\cN(t) \cB(t) = 0$ for all $t\in\R$. This implies
\begin{align*}
    N(t) \stackrel{\eqref{eq:ABC-trafo}}{=} [0, I_{n-pr}] U(t) BL(t) \stackrel{\eqref{eq:U(t)}}{=} \cN(t) BL(t) = \cN(t) \cB(t) \begin{smallbmatrix} I_m\\ 0\end{smallbmatrix} = 0
\end{align*}
for all $t\in\R$. It remains to show that $P_2 = \ldots = P_r = 0$. Fix $t\in\R$ and note that
$[P_1(t), \ldots, P_r(t), Q(t)] \stackrel{\eqref{eq:BIF}}{=} [0, I_{n-pr}] \hat A(t) \stackrel{\eqref{eq:ABC-trafo}}{=} [0, I_{n-pr}] (U(t)A+\dot U(t))U(t)^{-1} \stackrel{\eqref{eq:U(t)},\eqref{eq:U(t)inv}}{=} \big(\cN(t) A +\dot \cN(t)\big) [\cB(t) \big(\cC\cB(t)\big)^\dagger, V]$,
hence
\begin{align*}
   [P_1(t), \ldots, P_r(t)] \cC \cB(t) &\!=\! \big(\cN(t) A \!+\!\dot \cN(t)\big) \cB(t) \big(\cC\cB(t)\big)^\dagger \cC \cB(t) \\
    &\stackrel{\eqref{eq:B=BCBdagCB}}{=} \big(\cN(t) A +\dot \cN(t)\big) \cB(t) .
\end{align*}
Since $\cN(\cdot)\cB(\cdot)=0$ we find that $\ddt \big(\cN(\cdot)\cB(\cdot)\big)=0$, thus $\dot\cN(t) \cB(t) = - \cN(t) \dot \cB(t)$. Therefore, we have
\begin{align*}
 &[P_1(t), \ldots, P_r(t)] \cC \cB(t)  = \big(\cN(t) A +\dot \cN(t)\big) \cB(t) \\
   & = - \cN(t) \big(\ddt - A\big)\big(\cB(t)\big) \\
   & = - \cN(t) \left[\left(\ddt - A\right) \big(BL(t)\big), \ldots, \left(\ddt - A\right)^{r} \big(BL(t)\big)\right] \\
   & \stackrel{\cN(t)\cB(t)=0}{=} \left[0,\ldots,0, - \cN(t) \big(\ddt - A\big)^r\big(BL(t)\big)\right].
\end{align*}
We may infer, using~\eqref{eq:CB}, that $P_r(t) \Gamma L(t) = 0$, hence $P_r(t)=0$ since $\Gamma L(t)$ has full row rank~$p$ by~(P2). Successively we obtain $P_{r-1}(t) = \ldots = P_2(t) = 0$ and this finishes the proof of the theorem.
\end{proof}

The time-varying Byrnes-Isidori form from~\cite{IlchMuel07} has a certain uniqueness property as derived in~\cite[Thm.~B.7]{BergIlch15}. With the same proof we obtain the following result.

\begin{Cor}\label{Cor:unique}
Consider a system~\eqref{eq:ABC} with (P1) and (P2) such that $q=p$ for~$q$ in (P1). Then uniqueness of the entries in the normal form~\eqref{eq:BIF} holds as follows:
\begin{enumerate}
  \item the entries $[R_1, \ldots, R_r] = C A^r \cB\big(\cC \cB\big)^\dagger$ are uniquely defined;
  \item the time-varying linear (sub-)system $(Q, P_1, S)$ is unique up to $\big( (WQ + \dot W) W^{-1}, WP_1, SW^{-1}\big)$ for any Lyapunov transformation $W\in\cC^\infty(\R\to\Gl_{n-rm}(\R))$.
\end{enumerate}
\end{Cor}

We stress that the uniqueness property from Corollary~\ref{Cor:unique} is not true for $q>p$ in general. Consider~\eqref{eq:ABC} with $A=\begin{smallbmatrix} 0&1\\ 0&1\end{smallbmatrix}$, $B=\begin{smallbmatrix} 1&1\\ 1&3\end{smallbmatrix}$, $C=[1,0]$, $f=0$ and $L=I_2$, which is in the form~\eqref{eq:BIF} with $Q=1$. Computing the normal form~\eqref{eq:BIF} with~$U$ as in~\eqref{eq:U(t)} gives $\hat A = \begin{smallbmatrix} 2 & 1\\ -2 & 1\end{smallbmatrix}$, $\hat B = \begin{smallbmatrix} 1 & 1\\ -1 & 1\end{smallbmatrix}$, $\hat C = C$. Therefore, the new $Q$-block is given by $\hat Q = -1$, and it is straightforward to show that it cannot be obtained by a Lyapunov transformation of~$Q$. In the case $q>p$ it is thus important to follow exactly the construction procedure which leads to the transformation~$U$ in~\eqref{eq:U(t)}.


Also note that condition~\eqref{eq:cond-K-Gamma} is not always satisfied. Consider~\eqref{eq:ABC} with $A=\begin{smallbmatrix} 0&1&0&0\\ 0&0&1&0\\ 0&0&0&1\\ 0&0&0&0\end{smallbmatrix}$, $B=\begin{smallbmatrix} 0&0\\ 1&0\\ 0&1\\ 1&0\end{smallbmatrix}$, $C=[1,0,0,0]$, $f=0$ and $L=I_2$ as a counterexample.

\begin{Rem}\label{Rem:ZD}
An important system property in high-gain based adaptive control is a bounded-input, bounded-output property of the internal dynamics of the system, see e.g.~\cite{BergLe18a, IlchRyan08, IlchRyan02b}. In the case of linear time-invariant systems, this is implied by (but not equivalent to) asymptotic stability of the zero dynamics of the system; the latter property is extensively studied in the literature, see~\cite{ByrnWill84, Mare84}, and commonly known as the \emph{minimum phase} property, although this is not completely correct, see~\cite{IlchWirt13} and the references therein.

The usual assumption on system~\eqref{eq:ABC} would be that the zero dynamics of the linear part (ignoring the bounded nonlinearity) are uniformly exponentially stable. However, for system~\eqref{eq:ABC} with~$f=0$, a fixed output $y\in\cC^\infty(\R\to\R^p)$ does not uniquely define (up to initial values) a corresponding state~$x$ and input~$u$ since the actuator redundancy leads to $BL(t)$ not having full column rank in general. In other words, the zero dynamics are not necessarily \emph{autonomous}, cf.~\cite{BergIlch15}. To circumvent this problem we apply a rectangular input transformation~$u(t) = K(t) G(t)^{-1} v(t)$ to system~\eqref{eq:ABC} for~$K$ and~$G$ such that $\Gamma L(t) K(t) = G(t)$ and $N(t) K(t) = 0$. Since
\[
  BL(t) u(t) \!=\! BL(t) K(t) G(t)^{-1} v(t) \! =\! \cB(t) \big(\cC\cB(t)\big)^\dagger \begin{smallbmatrix} 0 \\ I_p\end{smallbmatrix} v(t)
\]
by Theorem~\ref{Thm:BIF}, this leads to the time-varying linear system $\left(A, \cB(\cdot) \big(\cC\cB(\cdot)\big)^\dagger \begin{smallbmatrix} 0 \\ I_p\end{smallbmatrix}, C\right)$ and hence we may assume that its zero dynamics are uniformly exponentially stable; this assumption is independent of the existence of~$K$ and~$G$, which is actually characterized by condition~\eqref{eq:cond-K-Gamma} since an additional invertible transformation does not change this condition: Assuming existence of $K\in\cC^\infty(\R\to\R^{m\times p})$ and $G\in\cC^\infty(\R\to{\Gl}_p(\R))$ such that $\Gamma L(t) K(t) = G(t)$ and $N(t) K(t) = 0$ for all $t\in\R$ leads to $\Gamma L(t) F(t) = I_p$ and $N(t) F(t) = 0$ for $F(t) = K(t) G(t)^{-1}$, $t\in\R$. We stress that by Theorem~\ref{Thm:BIF}, clearly~\eqref{eq:cond-K-Gamma} is always satisfied in the case $q=p$.

It is straightforward to show that there is a one-to-one correspondence between the zero dynamics of $\left(A, \cB(\cdot) \big(\cC\cB(\cdot)\big)^\dagger \begin{smallbmatrix} 0 \\ I_p\end{smallbmatrix}, C\right)$ and the solution set of the time-varying linear differential equation $\dot \eta(t) = Q(t) \eta(t)$ for~$Q$ as in~\eqref{eq:BIF}. Therefore, our assumption simplifies to assuming that $\dot \eta(t) = Q(t) \eta(t)$ is \emph{uniformly exponentially stable}, i.e., there exist $M, \mu>0$ such that for any solution $\eta\in\cC^1(\R\to\R^n)$ of $\dot \eta(t) = Q(t) \eta(t)$ we have $\|\eta(t)\| \le M e^{-\mu (t-t_0)}\|\eta(t_0)\|$ for all $t\ge t_0\ge 0$.
\end{Rem}

\begin{Ex}\label{Ex:runex1}
As a running example we consider a linearized model for the lateral motion of a Boeing 737 aircraft, which is taken from~\cite[Sec.~5.4]{TaoChen04}. The model is of the form~\eqref{eq:ABC} with $n=5$, $m=4$, $p=2$,
\begin{align*}
  A &= \begin{smallbmatrix} -0.13858 & 14.326 & -219.04 & 32.167 & 0\\ -0.02073 & -2.1692 & 0.91315 & 0.000256 & 0\\ 0.00289 & -0.16444 & -0.15768 & -0.00489 & 0\\ 0 & 1 & 0.00618 & 0 & 0\\ 0 & 0 & 1 & 0 & 0\end{smallbmatrix},\\
  B &= \begin{smallbmatrix} 0.15935 & 0.15935 & 0.00211 & 0.00211\\ 0.01264 & 0.01264 & 0.21326 & 0.21326\\ -0.12879 & -0.12879 & 0.00171 & 0.00171\\ 0 & 0 & 0 &0\\ 0 &0 &0&0\end{smallbmatrix},\quad C =\begin{smallbmatrix} 0 & 0 & 0 & 1 & 0\\ 0 & 0 & 0 & 0 & 1\end{smallbmatrix} \quad \text{and}\\
  x &= (v_b, p_b, r_b, \phi, \psi)^\top,\quad u = (d_{r1}, d_{r2}, d_{a1}, d_{a2})^\top.
\end{align*}
Here $v_b$ denotes the lateral velocity, $p_b$ the roll rate, $r_b$ the yaw rate, $\phi$ the roll
angle and $\psi$ the yaw angle. The roll angle $\phi$ and the yaw angle $\psi$ are chosen as outputs of the system. The inputs consist of the rudder position $d_{r1}+d_{r2}$ and the aileron position
$d_{a1}+d_{a2}$, so we see that we have two groups of actuators which are both double redundant. For this example we assume that no faults occur in the actuators $d_{r1}$ and $d_{a1}$, while $d_{r2}$ and $d_{a2}$ experience some faults, i.e.,
\[
    L(t) = \diag\big(1, l_2(t),1,l_4(t)\big),\quad f(t,u) = B \big( 0, f_2(t,u_2), 0, f_4(t,u_4)\big)^\top,
\]
where $u=(u_1,u_2,u_3,u_4)^\top$, for some bounded $l_2,l_4\in\cC^\infty(\R\to[0,1])$ (with bounded derivatives) and bounded $f_2,f_4\in\cC(\R^2\to\R)$ to be specified later. We compute the transformation matrix~$U$ as in~\eqref{eq:U(t)} and the normal form~\eqref{eq:BIF}. Since $CB = 0$ and $\Gamma = CAB = \begin{smallbmatrix} 0.01184 &   0.01184 &   0.21327 &  0.21327 \\ -0.12879 &  -0.12879& 0.00171 &  0.00171\end{smallbmatrix}$ has full row rank, we find that the relative degree of the system is $r=2$. Using MATLAB we calculate that, approximately,
\begin{align*}
   & U \!=\! \begin{smallbmatrix} 0& 0&  0 & 1 &  0\\ 0&0&0&0&1\\ 0&1& 0.00618&0&0\\ 0&0&1&0&0\\ 1 & -0.0198 &   1.23534 &  -14.16314& 219.17412\end{smallbmatrix}\!,\ \hat f(t,u) \!=\! \begin{bmatrix} 0\\ \Gamma\\ 0\end{bmatrix} \! \begin{smallpmatrix} 0\\ f_2(t,u_2)\\ 0\\ f_4(t,u_4)\end{smallpmatrix},\\
   & \hat A  \!=\! \begin{smallbmatrix} 0&0&1&0&0\\ 0&0&0&1&0\\ -0.29312&  4.53957 & -2.17063&   0.95118&  -0.02071\\ 0.03604& -0.63341& -0.16438&  -0.16023&  0.00289\\ 30.2546& 29.50071 & 0 & 0 & -0.1346\end{smallbmatrix},\ \  \hat B(t) = \begin{bmatrix} 0\\ \Gamma L(t)\\ 0\end{bmatrix}
\end{align*}
and $\hat C = [I_2, 0]$. In particular,~$U$ and~$\hat A$ are time-invariant and independent of~$l_2$ and~$l_4$, and~$U$ is a Lyapunov transformation. Furthermore, we find that~$Q$ as in Theorem~\ref{Thm:BIF} is given by $Q = -0.1346 \in \C_-$, hence $\dot \eta(t) = Q \eta(t)$ is exponentially stable.
\end{Ex}

%
%

\vspace*{-3mm}

%
\section{Fault tolerant control}\label{Sec:FTC}
%

\vspace*{-1mm}

\subsection{Preliminaries}\label{Ssec:FunCon_m=p}

Before we introduce a fault tolerant funnel controller we discuss some available results for nonlinear systems with arbitrary known relative degree and equal number of inputs and outputs without any faults. A funnel controller for such systems has been developed in~\cite{BergLe18a}. This controller is of the form
\begin{equation}\label{eq:fun-con}
\boxed{\begin{aligned}
e_0(t)&=e(t) = y(t) - y_{\rm ref}(t),\\
e_1(t)&=\dot{e}_0(t)+k_0(t)\,e_0(t),\\
e_2(t)&=\dot{e}_1(t)+k_1(t)\,e_1(t),\\
& \ \vdots \\
e_{r-1}(t)&=\dot{e}_{r-2}(t)+k_{r-2}(t)\,e_{r-2}(t),\\
k_i(t)&=1/(1-\varphi_i(t)^2\|e_i(t)\|^2),\quad i=0,\dots,r-1, \\
\end{aligned}
}
\end{equation}
with feedback law
\begin{equation}\label{eq:fb-m=p}
    u(t)= -k_{r-1}(t)\,e_{r-1}(t),
\end{equation}
where $y_{\rm ref}\in\mathcal{W}^{r,\infty}(\R_{\ge 0}\rightarrow \R^m)$ and $\varphi_i \in\Phi_{r-i}$ for $i=0,\ldots,r-1$. We stress that while the derivatives $\dot e_0,\ldots, \dot e_{r-2}$ appear in~\eqref{eq:fun-con}, they only serve as short-hand notations and may be resolved in terms of the tracking error, the funnel functions and the derivatives of these, cf.~\cite[Rem.~2.1]{BergLe18a}. After rewriting these variables, the controller is real-time capable.

The controller~\eqref{eq:fun-con},~\eqref{eq:fb-m=p} is shown to be feasible for a large class of nonlinear systems of the form
\begin{equation}\label{eq:nonlSys}
\begin{aligned}
y^{(r)}(t)&=f\big(d(t), T(y,\dot{y},\dots,y^{(r-1)})(t)\big)\\
&\quad + \Gamma\big(d(t), T(y,\dot{y},\dots,y^{(r-1)})(t)\big)\,u(t)\\
y|_{[-h,0]}&=y^0\in \mathcal{W}^{r-1,\infty}([-h,0]\rightarrow \R^m),
\end{aligned}
\end{equation}
where $h>0$ is the ``memory'' of the system, $r\in \N$ is the strict relative degree, and
\begin{itemize}
\item[(N1)] the disturbance satisfies $d\in \mathcal{L}^{\infty}(\R_{\ge 0}\rightarrow \R^p)$, $p\in \N$;
\item[(N2)] $f\in \mathcal{C}(\R^p\times \R^q\rightarrow \R^m),\ q\in \N$;
\item[(N3)] the high-frequency gain matrix function $\Gamma\in \mathcal{C}(\R^p\times \R^q\rightarrow \R^{m\times m})$  satisfies $\Gamma(d,\eta) + \Gamma(d,\eta)^\top > 0$ for all $(d,\eta)\in\R^p\times\R^q$;
\item[(N4)] $T:\mathcal{C}([-h,\infty)\rightarrow\R^{rm})\rightarrow \mathcal{L}_{\rm loc}^{\infty}(\R_{\ge 0}\rightarrow \R^q)$ is an operator with the following properties:
\begin{itemize}[a),leftmargin=0.5cm]
\item[a)] $T$ maps bounded trajectories to bounded trajectories, i.e, for all $c_1>0$, there exists $c_2>0$ such that for all $\zeta\in \mathcal{C}([-h,\infty)\rightarrow\R^{rm})$ with $\|\zeta(t)\|\le c_1$ for all $t\in [-h,\infty)$ we have $\sup_{t\in [0,\infty)}\|T(\zeta)(t)\|\le c_2$,
\item[b)] $T$ is causal, i.e, for all $t\ge 0$ and all $\zeta,\xi\in\mathcal{C}([-h,\infty)\rightarrow\R^{rm})$ with $\zeta|_{[-h,t)}=\xi|_{[-h,t)}$ we have $T(\zeta)|_{[0,t)}\overset{\rm a.a.}{=}T(\xi)|_{[0,t)}$, where ``a.a.'' stands for ``almost all''.
\item[c)] $T$ is locally Lipschitz continuous in the following sense: for all $t\ge 0$ and all $\xi\in\cC([-h,t]\to\R^{rm})$ there exist $\tau, \delta, c>0$ such that, for all $\zeta_1,\zeta_2\in \mathcal{C}([-h,\infty)\to\R^{rm})$ with $\zeta_i|_{[-h,t]}=\xi$ and $\|\zeta_i(s)-\xi(t)\|<\delta$ for all $s\in[t,t+\tau]$ and $i=1,2$, we have $\left\|\left(T(\zeta_1)-T(\zeta_2)\right)|_{[t,t+\tau]} \right\|_{\infty} \le c\left\|(\zeta_1-\zeta_2)|_{[t,t+\tau]}\right\|_{\infty}.$
\end{itemize}
\end{itemize}

In~\cite{BergLe18a,IlchRyan02b,IlchRyan09} it is shown that the class of systems~\eqref{eq:nonlSys} encompasses linear and nonlinear systems with strict relative degree and input-to-state stable internal dynamics and that the operator~$T$ allows for infinite-dimensional linear systems, systems with hysteretic effects or nonlinear delay elements, and combinations thereof.

In~\cite{BergLe18a}, the existence of global solutions of the initial value problem resulting from the application of the funnel controller~\eqref{eq:fun-con},~\eqref{eq:fb-m=p} to a system~\eqref{eq:nonlSys} is investigated. By a \emph{solution} of~\eqref{eq:fun-con}--\eqref{eq:nonlSys} on $[-h,\omega)$ we mean a function $y\in\cC^{r-1}([-h,\omega)\to\R^m)$, $\omega\in(0,\infty]$, with $\left.y\right|_{[-h,0]} = y^0$ such that $y^{(r-1)}|_{[0,\omega)}$ is weakly differentiable and satisfies the differential equation in~\eqref{eq:nonlSys} with~$u$ defined in~\eqref{eq:fb-m=p} for almost all $t\in[0,\omega)$; $y$ is called \emph{maximal}, if it has no right extension that is also a solution. Note that in~\cite{BergLe18a} a slightly stronger version of conditions~(N3) and~(N4)~c) is used. However, the proof does not change; in particular, regarding~(N4)~c), the existence part of the proof in~\cite{BergLe18a} relies on a result from~\cite{IlchRyan09} where the version from the present paper is used.

\vspace*{-3mm}

%
\subsection{Controller structure}\label{Ssec:ContStruc}
%

We introduce the fault tolerant funnel controller for systems of type~\eqref{eq:ABC} as an extension of the controller~\eqref{eq:fun-con},~\eqref{eq:fb-m=p} where we only change the feedback law~\eqref{eq:fb-m=p}. That is, the fault tolerant funnel controller consists of~\eqref{eq:fun-con} together with the new feedback law
\begin{equation}\label{eq:fb}
\boxed{
    u(t)= -k_{r-1}(t)\,K(t)\,e_{r-1}(t),
}
\end{equation}
where the reference signal and funnel functions have the following properties:
\begin{equation}\label{eq:con-ass-1}
\boxed{
\begin{aligned}
y_{\rm ref}&\in\, \mathcal{W}^{r,\infty}(\R_{\ge 0}\rightarrow \R^p),\\
\varphi_0&\in\, \Phi_r,\;\;
\varphi_1\in \Phi_{r-1},\;\ldots,\;\;
\varphi_{r-1}\in \Phi_{1}.\\
\end{aligned}
}
\end{equation}
We choose the bounded controller weight matrix function $K\in\cC^\infty(\R\to\R^{m\times p})$, if possible, such that
\begin{equation}\label{eq:con-ass-2}
\boxed{
\begin{aligned}
    & \exists\, \alpha>0:\ \Gamma L(t) K(t) + \big( \Gamma L(t) K(t) \big)^\top \geq \alpha I_p\ \\
    & \text{and}\quad N(t) K(t) = 0,\\
\end{aligned}
}
\end{equation}
where we use the notation from Theorem~\ref{Thm:BIF}. Note that condition~\eqref{eq:con-ass-2} is not always satisfied under the assumptions~(P1) and~(P2). Existence and possible choices for~$K$ are discussed in Subsection~\ref{Ssec:ConGain}. The first condition in~\eqref{eq:con-ass-2} is required to meet assumption~(N3) after a reformulation of the closed-loop system; the second condition is important to make the zero dynamics of~\eqref{eq:ABC} with the input transformation $u(t) = K(t) v(t)$ independent of the action of the new input~$v$, cf.\ Remark~\ref{Rem:ZD}. We stress that~\eqref{eq:fb} can be interpreted as~\eqref{eq:fb-m=p} multiplied with the controller weight~$K(t)$. 

In the sequel we investigate existence of solutions of the initial value problem resulting from the application of the funnel controller~\eqref{eq:fun-con},~\eqref{eq:fb} to a system~\eqref{eq:ABC}. Even if~\eqref{eq:ABC} is a linear system with~$f=0$ and~$L=I_m$, some care must be exercised with the existence of a solution of~\eqref{eq:ABC},~\eqref{eq:fun-con},~\eqref{eq:fb} since this closed-loop differential equation is time-varying, nonlinear and only defined on an open subset of $\R_{\ge 0}\times\R^n$. By a \emph{solution} of~\eqref{eq:ABC},~\eqref{eq:fun-con},~\eqref{eq:fb} on $[0,\omega)$ we mean a weakly differentiable function $x:[0,\omega)\to\R^n$, $\omega\in(0,\infty]$, which satisfies $x(0)=x^0$ and the differential equation in~\eqref{eq:ABC} with~$u$ defined in~\eqref{eq:fun-con},~\eqref{eq:fb} for almost all $t\in[0,\omega)$; $x$ is called \emph{maximal}, if it has no right extension that is also a solution.

\vspace*{-3mm}

%
\subsection{Feasibility of the controller}\label{Ssec:FeasCon}
%

We show feasibility of the controller~\eqref{eq:fun-con},~\eqref{eq:fb} for every system~\eqref{eq:ABC} which satisfies the assumptions~(P1),~(P2) and
\begin{enumerate}
  \item[(P3)] $U$ as in~\eqref{eq:U(t)} is a Lyapunov transformation,
  \item[(P4)] $\dot \eta(t) = Q(t) \eta(t)$ is uniformly exponentially stable for~$Q$ as in~\eqref{eq:BIF}.
\end{enumerate}

We stress that assumptions (P1)--(P4) and condition~\eqref{eq:con-ass-2} are only of structural nature and hold for a large class of systems; the controller design~\eqref{eq:fun-con},~\eqref{eq:fb} does not depend on the specific system parameters.

Note that~(P3) is satisfied, if~\eqref{eq:cond-boundU} holds. Under assumptions~(P1)--(P3) it follows from Theorem~\ref{Thm:BIF} that the transformation matrix~$U$ from~\eqref{eq:U(t)} can be used for a state space transformation as follows. Setting $z(t) := U(t) x(t)$ we obtain from~\eqref{eq:ABC} that
\begin{align*}
    \dot z(t) & =  \big(U(t)A(t)+\dot U(t)\big)U(t)^{-1} z(t) + U(t)BL(t) u(t)\\
    &\quad + U(t) f\big(t,U(t)^{-1}z(t),u(t)\big)
\end{align*}
and $y(t) = CU(t)^{-1} z(t)$. By Theorem~\ref{Thm:BIF} this implies that
\begin{equation*}\label{eq:hatABC}
  \dot z(t) = \hat A(t) z(t) + \hat B(t) u(t) + \hat f\big(t,z(t),u(t)\big),\  y(t) = \hat Cz(t)
\end{equation*}
and this is equivalent to
\begin{equation}\label{eq:hatABC-y-eta}
\begin{aligned}
  y^{(r)}(t) &= \sum\nolimits_{i=1}^r R_i(t) y^{(i-1)}(t) + S(t)\eta(t) + \Gamma L(t) u(t)\\
  &\quad + f_r\big(t,y(t),\ldots,y^{(r-1)}(t),\eta(t),u(t)\big),\\
  \dot \eta(t) &= \sum\nolimits_{i=1}^r P_i(t) y^{(i-1)}(t) + Q(t)\eta(t) + N(t) u(t)\\
  &\quad + f_\eta\big(t,y(t),\ldots,y^{(r-1)}(t),\eta(t),u(t)\big),
\end{aligned}
\end{equation}
where $z(t) = \big(y(t),\ldots,y^{(r-1)}(t),\eta(t)\big)$. By~(P4) it further follows that $\dot \eta(t) = Q(t) \eta(t)$ is uniformly exponentially. Together with~\eqref{eq:con-ass-2} these are the main ingredients for the proof of the following result.

\begin{Thm}\label{Thm:FunCon-1}
Consider a~system~\eqref{eq:ABC} which satisfies assumptions (P1)--(P4). Let $y_{\rm ref}, \varphi_0,\ldots,\varphi_{r-1}$ be as in~\eqref{eq:con-ass-1} and $x^0\in\R^n$ be an initial value such that $e_0,\ldots,e_{r-1}$ as defined in~\eqref{eq:fun-con} satisfy
\begin{equation}\label{eq:ini-err}
\varphi_i(0)\|e_i(0)\|<1\quad \text{ for }i=0,\ldots,r-1.
\end{equation}
Assume that there exists a bounded $K\in\cC^\infty(\R\to\R^{m\times p})$ such that~\eqref{eq:con-ass-2} is satisfied. Then the funnel controller~\eqref{eq:fun-con},~\eqref{eq:fb} applied to~\eqref{eq:ABC} yields an initial-value problem which has a solution, and every solution can be extended to a maximal solution $x:\left[0,\omega\right)\rightarrow \R^n$, $\omega\in(0,\infty]$, which satisfies:
\begin{enumerate}
\item The solution is global (i.e., $\omega=\infty$).
\item The signal $u:\R_{\ge0}\to\R^m$, $k_0,\ldots,k_{r-1}:\R_{\ge0}\to\R$ and $x:\R_{\ge0}\to\R^n$ are bounded.
\item The functions $e_0,\ldots,e_{r-1}:\R_{\ge0}\to\R^p$ evolve in their respective performance funnels and are uniformly bounded away from the funnel boundaries in the sense:
\begin{equation}\label{eq:error-bound}
\forall\, i=0,\dots,r-1\ \exists\, \varepsilon_i>0\ \forall\, t>0:\ \|e_i(t)\|\le \varphi_i(t)^{-1} - \varepsilon_i.
\end{equation}
In particular, the error $e(t)= y(t)-y_{\rm ref}(t)$ evolves in the funnel $\mathcal{F}_{\varphi_0}$ as in~\eqref{eq:perf_funnel} and stays uniformly away from its boundary.
\end{enumerate}
\end{Thm}
\begin{proof} We proceed in several steps.

\emph{Step 1}: We show existence of a solution of~\eqref{eq:ABC},~\eqref{eq:fun-con},~\eqref{eq:fb} and that it can be extended to a maximal solution. By assumptions~(P1)--(P3) it follows from Theorem~\ref{Thm:BIF} that the state space transformation $\big(y(t)^\top, \dot y(t)^\top,\ldots,y^{(r-1)}(t)^\top,\eta(t)^\top)^\top := U(t) x(t)$ puts system~\eqref{eq:ABC} into the form~\eqref{eq:hatABC-y-eta}. Set $v(t) := -k_{r-1}(t) e_{r-1}(t)$, then $u(t) = K(t) v(t)$. Using the same technique as in Step~1 of the proof of~\cite[Thm.~3.1]{BergLe18a} we find that there exist a relatively open set $\cD\subseteq \R_{\ge 0}\times \R^{rp}$ and $G:\cD\to \R^p$ such that
\[
    v(t) = - \frac{G\big(t,y(t), \dot y(t), \ldots,y^{(r-1)}(t)\big)}{1-\varphi^2_{r-1}(t) \|G\big(t,y(t),\dot y(t), \ldots,y^{(r-1)}(t)\big)\|^2}
\]
and $\big(0,y(0),\dot y(0),\ldots,y^{(r-1)}(0)\big) \in \cD$. Using the notation $Y(t) = \big(y(t),\dot y(t), \ldots,y^{(r-1)}(t)\big)$
and $\Xi(t) = \left(t,Y(t),\eta(t),\tfrac{-K(t) G(t,Y(t))}{1-\varphi^2_{r-1}(t) \|G(t,Y(t))\|^2}\right)$, the closed-loop system~\eqref{eq:ABC},~\eqref{eq:fun-con},~\eqref{eq:fb} can be reformulated as, invoking that $N(t) K(t) = 0$ by~\eqref{eq:con-ass-2},
\begin{equation}\label{eq:CL}
\begin{aligned}
  y^{(r)}(t) &\!=\! \sum\nolimits_{i=1}^r R_i(t) y^{(i-1)}(t) \!+\! S(t)\eta(t) \!-\!  \tfrac{\Gamma L(t) K(t) G(t,Y(t))}{1-\varphi^2_{r-1}(t) \|G(t,Y(t))\|^2} \!+\! f_r\left(\Xi(t)\right)\\
   \dot \eta(t) &= \sum\nolimits_{i=1}^r P_i(t) y^{(i-1)}(t) + Q(t)\eta(t) + f_\eta\left(\Xi(t)\right).
\end{aligned}
\end{equation}
It is clear that~\eqref{eq:CL} can be reformulated as a first-order system
\begin{align*}
    \ddt \begin{smallpmatrix} Y(t)\\ \eta(t)\end{smallpmatrix} &= F\left(t,\begin{smallpmatrix} Y(t)\\ \eta(t)\end{smallpmatrix}\right),\
    \begin{smallpmatrix} Y(0)\\ \eta(0)\end{smallpmatrix} = U(0) x^0,
\end{align*}
with a suitable continuous function $F: \cD\times\R^{n-rp}\to \R^n$. Furthermore, $(0,U(0)x^0)\in \cD\times\R^{n-rp}$ and $\cD\times\R^{n-rp}$ is relatively open in $\R_{\ge 0}\times\R^n$. Hence, by~\cite[\S\,10, Thm.~XX]{Walt98} there exists a weakly differentiable solution of~\eqref{eq:CL} satisfying the initial conditions and every solution can be extended to a maximal solution; let $(Y,\eta):[0,\omega)\to\R^{n}$, $\omega\in(0,\infty]$, be such a maximal solution.

\emph{Step 2}: We show that $(Y,\eta)$ also solves a closed-loop system which is of the form~\eqref{eq:nonlSys},~\eqref{eq:fun-con} and $u(t) = v(t)$ in~\eqref{eq:nonlSys}. Set $d_1(t) := f_\eta\left(\Xi(t)\right)$
for $t\in[0,\omega)$ and let $\Phi_Q(\cdot,\cdot)$ be the transition matrix of the linear time-varying system $\dot \eta(t) = Q(t) \eta(t)$. Then the variation of constants formula yields that
\begin{align*}
    \eta(t) & = \Phi_Q(t,0)\eta(0) + \int_0^t \Phi_Q(t,s) \left( \sum\nolimits_{i=1}^r P_i(s) y^{(i-1)}(s) + d_1(s)\right) \ds{s}
\end{align*}
for all $t\in[0,\omega)$. Set
\begin{align*}
   d(t) &:= S(t)\Phi_Q(t,0) \eta(0) + \int_0^t S(t) \Phi_Q(t,s) d_1(s) \ds{s} + f_r\left(\Xi(t)\right)
\end{align*}
for $t\in[0,\omega)$. Since $S$, $f_\eta$ and $f_r$ are bounded and $\dot \eta(t) = Q(t) \eta(t)$ is uniformly exponentially by assumptions~(P1)--(P4), it follows that~$d$ is bounded on $[0,\omega)$. If $\omega<\infty$, we define $d(t):=0$ for $t\ge \omega$ and obtain $d\in\cL^\infty(\R_{\ge 0}\to\R^p)$. Define the operator $T:\mathcal{C}([0,\infty)\rightarrow\R^{rp})\rightarrow \mathcal{L}_{\rm loc}^{\infty}(\R_{\ge 0}\rightarrow \R^p)$ by
\begin{align*}
  T(\zeta_1,\ldots,\zeta_r)(t) \!=\! \sum\nolimits_{i=1}^r  R_i(t) \zeta_i(t) \!+\!  \sum\nolimits_{i=1}^r \int_0^t S(t) \Phi_Q(t,s)  P_i(s) \zeta_i(s) \ds{s}
\end{align*}
for $t\ge 0$. Then we have $y^{(r)}(t) = T(Y)(t) + d(t) + \Gamma L(t) K(t) v(t)$ for almost all $t\in[0,\omega)$. Finally, we seek a function $g\in\cC(\R^\ell\to\R^{p\times p})$, $\ell\in\N$, and a bounded function $\bar d\in\cC^\infty(\R\to\R^\ell)$ such that $g\big(\bar d(t)\big) = \Gamma L(t) K(t)$ for all $t\ge 0$ and $g(x) + g(x)^\top > 0$ for all $x\in\R^\ell$. The construction is as follows: By assumption~\eqref{eq:con-ass-2} we have that $A (t) := \Gamma L(t) K(t) + \big( \Gamma L(t) K(t) \big)^\top - \tfrac{\alpha}{2} I_p > 0$ for all $t\ge 0$, hence there exists a pointwise Cholesky decomposition $A(t) = H(t)H(t)^\top$. For $x=(x_{11},\ldots,x_{1p}, x_{21},\ldots,x_{pp})^\top \in \R^{p^2}$ set $\cM(x) = \begin{smallbmatrix} x_{11} & \cdots & x_{1p}\\ \vdots && \vdots\\ x_{p1} & \cdots & x_{pp}\end{smallbmatrix}$ and define, for $\alpha>0$ as above,
\begin{align*}
    g_1:\ &\R^{p^2}\to\R^{p\times p},\ x \mapsto \tfrac12 \cM(x) \cM(x)^\top + \tfrac{\alpha}{4} I_p,\\
    g_2:\ &\R^{p^2}\to\R^{p\times p},\ x \mapsto \tfrac12 \left(\cM(x) - \cM(x)^\top \right).
\end{align*}
Let $H(t) = \big(h_{ij}(t)\big)_{i,j=1,\ldots,p}$ and $\Gamma L(t) K(t) = \big(k_{ij}(t)\big)_{i,j=1,\ldots,p}$, then
\begin{align*}
    &g_1\big(h_{11}(t),\ldots,h_{pp}(t)\big) &=\ & \tfrac12 H(t) H(t)^\top  + \tfrac{\alpha}{4} I_p\\
    && =\ & \tfrac12\left( \Gamma L(t) K(t) + \big( \Gamma L(t) K(t) \big)^\top \right),\\
& g_2\big(k_{11}(t),\ldots,k_{pp}(t)\big) &=\ &  \tfrac12\left( \Gamma L(t) K(t) - \big( \Gamma L(t) K(t) \big)^\top \right),\quad \text{thus}\\
&g_1\big(h_{11}(t),\ldots,h_{pp}(t)\big) &+\ & g_2\big(k_{11}(t),\ldots,k_{pp}(t)\big)= \Gamma L(t) K(t).
\end{align*}
Define $g:\R^{p^2}\times\R^{p^2}\to\R^{p\times p}, (x,z)\mapsto g_1(x) + g_2(z)$, then we find that $g(x,z) + g(x,z)^\top  = g_1(x) + g_1(x)^\top + g_2(z) + g_2(z)^\top = g_1(x) + g_1(x)^\top \ge \tfrac{\alpha}{2} I_p > 0$ for all $x, z\in\R^{p^2}$. With the bounded function
\[
    \bar d(\cdot) := \big(h_{11}(\cdot),\ldots,h_{pp}(\cdot), k_{11}(\cdot),\ldots,k_{pp}(\cdot)\big) 
\]
we finally obtain that the solution $(Y,\eta)$ from Step~1 satisfies
\begin{equation}\label{eq:CL2}
    y^{(r)}(t) = T(Y)(t) + d(t) + g\big(\bar d(t)\big) v(t)
\end{equation}
for almost all $t\in[0,\omega)$, where the input $v(t) = -k_{r-1}(t) e_{r-1}(t)$ is obtained from the controller~\eqref{eq:fun-con}.\\
Invoking boundedness of~$d$ and~$\bar d$, system~\eqref{eq:CL2} satisfies assumptions~(N1) and~(N2). Assumption~(N3) is a consequence of the construction of~$g$. The operator~$T$ is clearly causal and locally Lipschitz. By~(P4), $T$ maps bounded trajectories to bounded trajectories and therefore~\eqref{eq:CL2} satisfies condition~(N4).

\emph{Step 3}:  By Steps~1 and~2, the maximal solution $(Y,\eta)$ is also a solution of~\eqref{eq:CL2} with~\eqref{eq:fun-con} and  $v(t) = -k_{r-1}(t) e_{r-1}(t)$, hence~\cite[Thm.~3.1]{BergLe18a} yields that it can be extended to a global solution, i.e., $\omega = \infty$. Statements~(ii) and~(iii) are consequences of~\cite[Thm.~3.1]{BergLe18a} as well.
\end{proof}

\begin{Rem}
We like to point out that a drawback of the controller design~\eqref{eq:fun-con},~\eqref{eq:fb}, which still needs to be resolved, is that the derivatives of the output must be available for the controller. However, there are several applications where this condition is not satisfied, and it may even be hard to obtain suitable estimates of the output derivatives, in particular in the presence of measurement noise. A first approach to treat these problems using a ``funnel pre-compensator'' has been developed in~\cite{BergReis18b, BergReis18a} for systems with relative degree $r\in\{2,3\}$, and was successfully applied to multibody systems in~\cite{BergOtto19}.
\end{Rem}

\vspace*{-3mm}

%
\subsection{Discussion of controller weight matrix}\label{Ssec:ConGain}
%

We discuss possible choices for the bounded controller weight matrix~$K\in\cC^\infty(\R\to\R^{m\times p})$ satisfying~\eqref{eq:con-ass-2}. We distinguish the two cases $\rk BL(t) = p$ and $\rk BL(t) = q > p$. In practical applications, it is frequently the case that some actuators are used to perform similar control tasks or they can be divided into~$p$ groups of actuators with the same physical characteristics, where~$p$ is the number of outputs, see e.g.~\cite{TaoChen04}. Due to this redundancy it may be assumed (and actually is quite probable) that in each group at least one actuator remains (partially) functional, i.e., does not experience a total fault. This means that we are in the case $\rk BL(t) = p$. An interesting and relevant example is mentioned in~\cite[p.~103]{TaoChen04}.

If there are~$q$ groups of actuators and~$p$ outputs with $q>p$, then the system typically has an unnecessary high redundancy. When it is still possible to guarantee that at least one actuator without total fault remains in each group, then complete groups of actuators may be switched off so that $q=p$ is achieved. 

\subsubsection{The case $\rk BL(t) = p$}

Under the assumptions~(P1)--(P3) it follows from Theorem~\ref{Thm:BIF} that in the case $q=p$ we have $N=0$, so the second condition in~\eqref{eq:con-ass-2} is satisfied for any choice of~$K$. In order to satisfy the first condition in~\eqref{eq:con-ass-2}, a possible choice is $K(t) = \Gamma^\top$ and the requirement that there exists $\alpha>0$ such that $\Gamma \big(L(t) + L(t)^\top\big) \Gamma^\top \ge \alpha I_p$ for all $t\in\R$. This condition means that we have at least~$p$ linearly independent actuators, the reliability of which does not converge to zero. In other words, in each group of actuators at least one remains functional, see the discussion above. Clearly, for this specific choice of~$K$ we have to assume that the high frequency gain matrix~$\Gamma$ of~\eqref{eq:ABC} is known; apart from that, no knowledge of the system parameters is required.

\subsubsection{The case $\rk BL(t) = q > p$}

Under the assumptions~(P1)--(P3) it follows from Theorem~\ref{Thm:BIF} and Remark~\ref{Rem:ZD} that in the case $q>p$ there exists~$K$ such that $\Gamma L(t) K(t)$ is invertible and $N(t) K(t) = 0$ for all $t\in\R$ if, and only if, condition~\eqref{eq:cond-K-Gamma} is satisfied. In this case, $K(t)$ as in~\eqref{eq:K(t)=} is a feasible choice which satisfies $\Gamma L(t) K(t) = I_p$ and hence~\eqref{eq:con-ass-2} holds true. However, this requires knowledge of the system parameters and of the reliability matrix function~$L$ from~(P1).

\vspace*{-2.5mm}
%
\section{Simulation}\label{Sec:Sim}
%

We illustrate the fault tolerant funnel controller~\eqref{eq:fun-con},~\eqref{eq:fb} by applying it to the model of the Boeing 737 aircraft from Example~\ref{Ex:runex1}. As reference trajectories we choose $y_{\rm ref,1}(t) = 2\sin t$, $y_{\rm ref,2}(t) = \cos t$, the initial value is $x(0) = 0$, and the funnel functions are $\varphi_0(t) = (5 e^{-t} + 0.1)^{-1}$ and $\varphi_1(t) = \big(\tfrac52 e^{-0.5\, t} + 0.1\big)^{-1}$, hence~\eqref{eq:con-ass-1} is satisfied. Obviously, the initial errors lie within the respective funnel boundaries, i.e.,~\eqref{eq:ini-err} is satisfied. The controller weight matrix is chosen as $K(t) = \Gamma^\top$. For the simulation, we assume that the actuator $d_{r2}$ has a slowly decreasing efficiency to $50\%$ of the original capability on the time interval $[0,6]$ and at $t=6$ another fault occurs so that we have an actuator saturation by~1 (which means an effective saturation by~$0.5$ due to the $50\%$ reduction of efficiency). Using the smooth error function \textit{erf} and the complementary error function \textit{erfc} (note that all derivatives of \textit{erf} and \textit{erfc} are bounded), which are implemented in MATLAB, this behavior can be modelled by
\begin{align*}
  l_2(t) &= \tfrac14 \erfc(t-3)+ \tfrac14 \erfc\big(100(t-6)\big),\\
  f_2(t,u_2) &= \tfrac14 \big(1+\erf\big(100(t-6)\big)\big) \cdot \sat_1(u_2),
\end{align*}
where $\sat_1(v) = \sgn(v)$ for $|v|\ge 1$ and $\sat_1(v) = v$ for $|v| < 1$. We further assume that the actuator $d_{a2}$ has a sudden total fault at $t=7$, which can be modelled by $l_4(t) = \tfrac12 \erfc\big(20(t-7)\big)$ and $f_4(t,u_4) = 0$. After the faults, the effective input actions are $u_1(t) = d_{r1}(t)$, $u_2(t) = l_2(t) d_{r2}(t) + f_2\big(t,d_{r2}(t)\big)$, $u_3(t) = d_{a1}(t)$ and $u_4(t) = l_4(t) d_{a2}(t)$.

Since the actuators $d_{r1}$ and $d_{a1}$ are assumed to experience no faults, condition~\eqref{eq:con-ass-2} is clearly satisfied. It further follows from Example~\ref{Ex:runex1} that (P1)--(P4) are satisfied. Therefore, fault tolerant funnel control is feasible by Theorem~\ref{Thm:FunCon-1}. The simulation of the controller~\eqref{eq:fun-con},~\eqref{eq:fb} applied to the model of the Boeing 737 aircraft from Example~\ref{Ex:runex1} over the time interval $[0,10]$ has been performed in MATLAB (solver: {\tt ode45}, rel.\ tol.: $10^{-14}$, abs.\ tol.: $10^{-10}$) and is depicted in Fig.~\ref{fig:sim}. The decreasing efficiency of~$u_2$ can clearly be seen as well as that it is saturated by~$0.5$ on the interval $[6,10]$; the saturation is active on the interval $[7,9]$. Furthermore, it can be seen that at $t=7$, $u_4$ has a total fault and hence~$u_3$ needs to increase in order to compensate for this. The tracking performance is not affected at all by these faults; the tracking errors evolve within the prescribed performance funnels.

\vspace{-5mm}
\captionsetup[subfloat]{labelformat=empty}
\begin{figure}[h!tb]
  \centering
  \subfloat[Fig.~\ref{fig:sim}a: Funnel and tracking errors]
{
\centering
  \includegraphics[width=7.4cm]{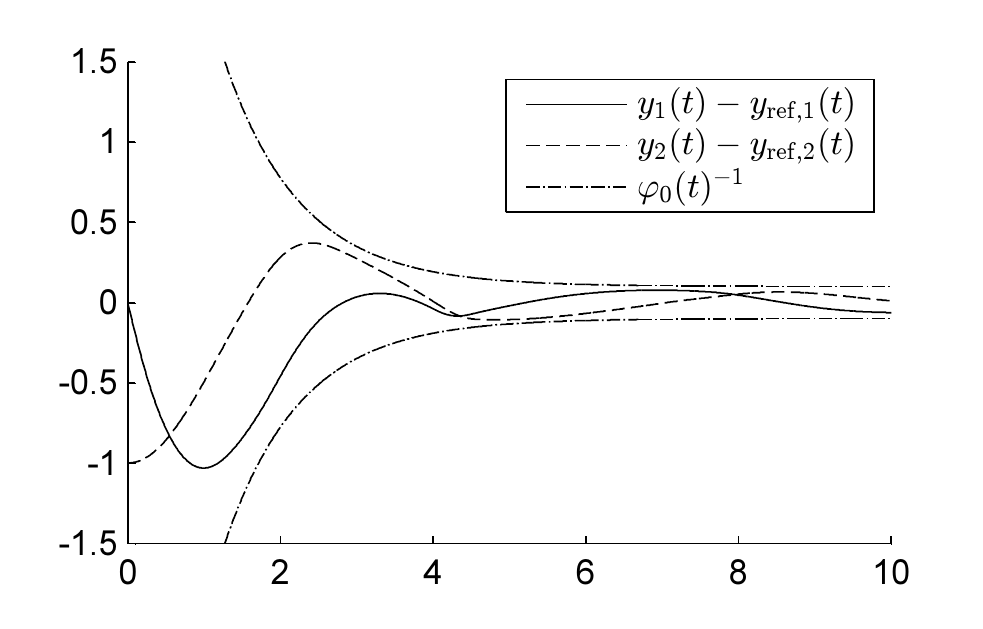}
\label{fig:sim-e}
}\\[-1mm]
\subfloat[Fig.~\ref{fig:sim}b: Input functions]
{
\centering
 \hspace*{-5mm} \includegraphics[width=7.4cm]{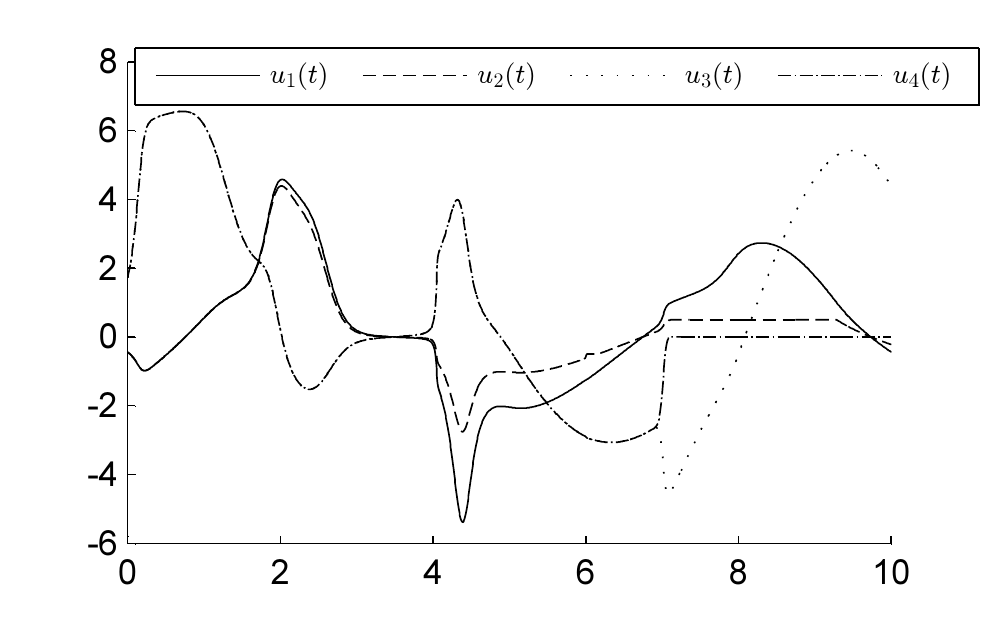}
\label{fig:sim-u}
}
\vspace*{-1mm}
\caption{Simulation of the controller~\eqref{eq:fun-con},~\eqref{eq:fb} for the Boeing 737 aircraft.}
\label{fig:sim}
\end{figure}

\vspace*{-7mm}

\bibliographystyle{ieeetr}

\begin{thebibliography}{10}

\bibitem{TaoChen04}
G.~Tao, S.~Chen, X.~Tang, and S.~M. Joshi, {\em Adaptive Control of Systems
  with Actuator Failures}.
\newblock London: Springer-Verlag, 2004.

\bibitem{ZhanJian08}
Y.~Zhang and J.~Jiang, ``Bibliographical review on reconfigurable
  fault-tolerant control systems,'' {\em Ann. Rev. Contr.}, vol.~32,
  pp.~229--252, 2008.

\bibitem{GaoCeca15a}
Z.~Gao, C.~Cecati, and S.~X. Ding, ``A survey of fault diagnosis and
  fault-tolerant techniques--{P}art {I}: Fault diagnosis with model-based and
  signal-based approaches,'' {\em {IEEE} Trans. Ind. Electron.}, vol.~62,
  no.~6, pp.~3757--3767, 2015.

\bibitem{GaoCeca15b}
Z.~Gao, C.~Cecati, and S.~X. Ding, ``A survey of fault diagnosis and
  fault-tolerant techniques--{P}art {II}: Fault diagnosis with knowledge-based
  and hybrid/active approaches,'' {\em {IEEE} Trans. Ind. Electron.}, vol.~62,
  no.~6, pp.~3768--3774, 2015.

\bibitem{DengYang16}
C.~Deng and G.-H. Yang, ``Cooperative adaptive output feedback control for
  nonlinear multi-agent systems with actuator failures,'' {\em Neurocomputing},
  vol.~199, pp.~50--57, 2016.

\bibitem{XieYang16a}
C.-H. Xie and G.-H. Yang, ``Data-based fault-tolerant control for affine
  nonlinear systems with actuator faults,'' {\em {ISA} {T}ransactions},
  vol.~64, pp.~285--292, 2016.

\bibitem{ZhanYang17b}
J.-X. Zhang and G.-H. Yang, ``Prescribed performance fault-tolerant control of
  uncertain nonlinear systems with unknown control directions,'' {\em {IEEE}
  Trans. Autom. Control}, vol.~62, no.~12, pp.~6529--6535, 2017.

\bibitem{LiYang12}
X.-J. Li and G.-H. Yang, ``Robust adaptive fault-tolerant control for uncertain
  linear systems with actuator failures,'' {\em {IET} Control Theory Appl.},
  vol.~6, no.~10, pp.~1544--1551, 2012.

\bibitem{YangYang00}
Y.~Yang, G.-H. Yang, and Y.~C. Soh, ``Reliable control of discrete-time systems
  with actuator failure,'' {\em {IEE} Proceedings Control Theory and
  Applications}, vol.~147, no.~4, pp.~428--432, 2000.

\bibitem{TaoJosh01}
G.~Tao, S.~M. Joshi, and X.~Ma, ``Adaptive state feedback and tracking control
  of systems with actuator failures,'' {\em {IEEE} Trans. Autom. Control},
  vol.~46, no.~1, pp.~78--95, 2001.

\bibitem{BechRovi08}
C.~P. Bechlioulis and G.~A. Rovithakis, ``Robust adaptive control of feedback
  linearizable {MIMO} nonlinear systems with prescribed performance,'' {\em
  {IEEE} Trans. Autom. Control}, vol.~53, no.~9, pp.~2090--2099, 2008.

\bibitem{BechRovi14}
C.~P. Bechlioulis and G.~A. Rovithakis, ``A low-complexity global
  approximation-free control scheme with prescribed performance for unknown
  pure feedback systems,'' {\em Automatica}, vol.~50, no.~4, pp.~1217--1226,
  2014.

\bibitem{IlchRyan02b}
A.~Ilchmann, E.~P. Ryan, and C.~J. Sangwin, ``Tracking with prescribed
  transient behaviour,'' {\em ESAIM: Control, Optimisation and Calculus of
  Variations}, vol.~7, pp.~471--493, 2002.

\bibitem{IlchRyan08}
A.~Ilchmann and E.~P. Ryan, ``High-gain control without identification: a
  survey,'' {\em GAMM Mitt.}, vol.~31, no.~1, pp.~115--125, 2008.

\bibitem{IlchTren04}
A.~Ilchmann and S.~Trenn, ``Input constrained funnel control with applications
  to chemical reactor models,'' {\em Syst. Control Lett.}, vol.~53, no.~5,
  pp.~361--375, 2004.

\bibitem{Hack17}
C.~M. Hackl, {\em Non-identifier Based Adaptive Control in Mechatronics--Theory
  and Application}, vol.~466 of {\em Lecture Notes in Control and Information
  Sciences}.
\newblock Cham, Switzerland: Springer-Verlag, 2017.

\bibitem{BergOtto19}
T.~Berger, S.~Otto, T.~Reis, and R.~Seifried, ``Combined open-loop and funnel
  control for underactuated multibody systems,'' {\em Nonlinear Dynamics},
  vol.~95, pp.~1977--1998, 2019.

\bibitem{SenfPaug14}
A.~Senfelds and A.~Paugurs, ``Electrical drive {DC} link power flow control
  with adaptive approach,'' in {\em Proc. 55th Int. Sci. Conf. Power Electr.
  Engg. Riga Techn. Univ., Riga, Latvia}, pp.~30--33, 2014.

\bibitem{BergReis14a}
T.~Berger and T.~Reis, ``Zero dynamics and funnel control for linear electrical
  circuits,'' {\em J. Franklin Inst.}, vol.~351, no.~11, pp.~5099--5132, 2014.

\bibitem{PompWeye15}
A.~Pomprapa, S.~Weyer, S.~Leonhardt, M.~Walter, and B.~Misgeld, ``Periodic
  funnel-based control for peak inspiratory pressure,'' in {\em Proc.
  54th~{IEEE} Conf. Decis. Control, Osaka, Japan}, pp.~5617--5622, 2015.

\bibitem{BergRaue18}
T.~Berger and A.-L. Rauert, ``A universal model-free and safe adaptive cruise
  control mechanism,'' in {\em Proceedings of the MTNS 2018}, (Hong Kong),
  pp.~925--932, 2018.

\bibitem{BergLe18a}
T.~Berger, H.~H. L{\^e}, and T.~Reis, ``Funnel control for nonlinear systems
  with known strict relative degree,'' {\em Automatica}, vol.~87, pp.~345--357,
  2018.

\bibitem{ZhaoJian98}
Q.~Zhao and J.~Jiang, ``Reliable state feedback control system design against
  actuator failures,'' {\em Automatica}, vol.~34, no.~10, pp.~1267--1272, 1998.

\bibitem{Gert88}
J.~J. Gertler, ``Survey of model-based failure detection and isolation in
  complex plants,'' {\em {IEEE} Control Systems Magazine}, vol.~8, no.~6,
  pp.~3--11, 1988.

\bibitem{DeLuMatt03}
A.~De~Luca and R.~Mattone, ``Actuator failure detection and isolation using
  generalized momenta,'' in {\em Proc.~2003 {IEEE} Int. Conf. Robotics Autom.},
  (Taipei, Taiwan), pp.~634--639, 2003.

\bibitem{IlchMuel07}
A.~Ilchmann and M.~Mueller, ``Time-varying linear systems: Relative degree and
  normal form,'' {\em {IEEE} Trans. Autom. Control}, vol.~52, no.~5,
  pp.~840--851, 2007.

\bibitem{Berg16b}
T.~Berger, ``Zero dynamics and funnel control of general linear
  differential-algebraic systems,'' {\em {ESAIM} Control Optim. Calc. Var.},
  vol.~22, no.~2, pp.~371--403, 2016.

\bibitem{Isid95}
A.~Isidori, {\em Nonlinear Control Systems}.
\newblock Communications and Control Engineering Series, Berlin:
  Springer-Verlag, 3rd~ed., 1995.

\bibitem{Muel09a}
M.~Mueller, ``Normal form for linear systems with respect to its vector
  relative degree,'' {\em Linear Algebra Appl.}, vol.~430, no.~4,
  pp.~1292--1312, 2009.

\bibitem{KunkMehr06}
P.~Kunkel and V.~Mehrmann, {\em Differential-Algebraic Equations. Analysis and
  Numerical Solution}.
\newblock Z{\"u}rich, Switzerland: EMS Publishing House, 2006.

\bibitem{BergIlch15}
T.~Berger, A.~Ilchmann, and F.~Wirth, ``Zero dynamics and stabilization for
  analytic linear systems,'' {\em Acta Applicandae Mathematicae}, vol.~138,
  no.~1, pp.~17--57, 2015.

\bibitem{ByrnWill84}
C.~I. Byrnes and J.~C. Willems, ``Adaptive stabilization of multivariable
  linear systems,'' in {\em Proc. 23rd~{IEEE} Conf. Decis. Control},
  pp.~1574--1577, 1984.

\bibitem{Mare84}
I.~M.~Y. Mareels, ``A simple selftuning controller for stably invertible
  systems,'' {\em Syst. Control Lett.}, vol.~4, no.~1, pp.~5--16, 1984.

\bibitem{IlchWirt13}
A.~Ilchmann and F.~Wirth, ``On minimum phase,'' {\em Automatisierungstechnik},
  vol.~12, pp.~805--817, 2013.

\bibitem{IlchRyan09}
A.~Ilchmann and E.~P. Ryan, ``Performance funnels and tracking control,'' {\em
  Int. J. Control}, vol.~82, no.~10, pp.~1828--1840, 2009.

\bibitem{Walt98}
W.~Walter, {\em Ordinary Differential Equations}.
\newblock New York: Springer-Verlag, 1998.

\bibitem{BergReis18b}
T.~Berger and T.~Reis, ``The {F}unnel {P}re-{C}ompensator,'' {\em Int. J.
  Robust \& Nonlinear Control}, vol.~28, no.~16, pp.~4747--4771, 2018.

\bibitem{BergReis18a}
T.~Berger and T.~Reis, ``Funnel control via funnel pre-compensator for minimum
  phase systems with relative degree two,'' {\em {IEEE} Trans. Autom. Control},
  vol.~63, no.~7, pp.~2264--2271, 2018.

\end{thebibliography}

%
%
%

\end{document}